\documentclass[11pt]{article}
\usepackage[utf8]{inputenc}
\usepackage{bbm}
\usepackage{amsmath}
\usepackage{csquotes}
\usepackage{mathtools}
\usepackage{amsthm}
\usepackage{amssymb}
\usepackage{pdflscape}
\usepackage{amsfonts}
\usepackage{anyfontsize}
\usepackage[english]{babel}
\usepackage{xfrac}
\usepackage[hang]{footmisc}
\usepackage{lipsum}
\usepackage{enumitem}
\usepackage{hyperref}
\hypersetup{colorlinks=true, pdfauthor=author, allcolors=blue}
\usepackage{multirow}
\usepackage{graphicx,xcolor,textpos}
\usepackage{natbib}
\usepackage{emptypage}
\usepackage{appendix}
\usepackage[toc,acronym,nonumberlist]{glossaries}
\usepackage{diagbox}
\usepackage{caption}
\usepackage{subcaption}
\usepackage{float}
\usepackage{listings}
\usepackage{tikz}
\usetikzlibrary{shapes.geometric, arrows}
\tikzstyle{arrow} = [thick,->,>=stealth]
\tikzstyle{process} = [rectangle, minimum width=3cm, minimum height=1cm, text centered, draw=black, fill=orange!30]
\tikzstyle{decision} = [rectangle, minimum width=3cm, minimum height=1cm, text centered, draw=black, fill=blue!30]
\tikzstyle{observation} = [rectangle, minimum width=3cm, minimum height=1cm, text centered, draw=black, fill=yellow!30]

\DeclareMathOperator{\EX}{\mathbb{E}}

\DeclareMathOperator*{\argmax}{arg\,max}

\newcommand{\diff}{\mathop{}\!d}

\newcommand{\indep}{\perp \!\!\! \perp}

\newcommand*{\QEDB}{\null\nobreak\hfill\ensuremath{\square}}%

\newtheorem{theorem}{Theorem}
\newtheorem{lemma}{Lemma}

\begin{document}
	
\title{An instrumental variable approach under dependent censoring}

\author{$\text{Gilles Crommen}^1$ \\ \url{gilles.crommen@kuleuven.be} \and $\text{Jad Beyhum}^{1,2}$ \\ \url{jad.beyhum@gmail.com} \and $\text{Ingrid Van Keilegom}^1$ \\ \url{ingrid.vankeilegom@kuleuven.be}}

\date{$^1$ORSTAT, KU Leuven, Naamsestraat 69, Leuven, 3000, Belgium \\ $^2$CREST, ENSAI, 51 Rue Blaise Pascal, Bruz, 35170, France}

\maketitle

\noindent \textbf{Acknowledgments:} Financial support from the European Research Council (2016-2022, Horizon 2020 / ERC grant agreement No. 694409) is gratefully acknowledged. The authors are grateful to Sara Rutten and Ilias Willems for their useful comments.

\begin{abstract}
\noindent This paper considers the problem of inferring the causal effect of a variable $Z$ on a dependently censored survival time $T$. We allow for unobserved confounding variables, such that the error term of the regression model for $T$ is correlated with the confounded variable $Z$. Moreover, $T$ is subject to dependent censoring. This means that $T$ is right censored by a censoring time $C$, which is dependent on $T$ (even after conditioning out the effects of the measured covariates). A control function approach, relying on an instrumental variable, is leveraged to tackle the confounding issue. Further, it is assumed that $T$ and $C$ follow a joint regression model with bivariate Gaussian error terms and an unspecified covariance matrix such that the dependent censoring can be handled in a flexible manner. Conditions under which the model is identifiable are given, a two-step estimation procedure is proposed, and it is shown that the resulting estimator is consistent and asymptotically normal. Simulations are used to confirm the validity and finite-sample performance of the estimation procedure. Finally, the proposed method is used to estimate the causal effect of job training programs on unemployment duration.
\end{abstract}

\bigskip \noindent \small \textbf{Keywords}: dependent censoring, causal inference, instrumental variable, control function, survival analysis. \normalsize

\section{Introduction}
When estimating the effect of a variable $Z$ on a censored survival time $T$, unmeasured confounding can be a possible source of bias. Under certain assumptions, the instrumental variable (IV) approach allows us to negate this bias without having to observe any of the confounding variables. Within survival analysis, IV methods have recently been receiving increased attention to estimate causal effects on censored outcomes. However, almost all of these approaches assume that the censoring time is (conditionally) independent of the survival time. In this work, we propose an IV method that can identify causal effects while allowing for dependent censoring. Precisely, let $T$ depend log-linearly on a vector of observed covariates $X$, a confounded variable $Z$ and some error term, denoted by $u_T$, which represents unobserved heterogeneity. There is a confounding issue when $Z$ and $u_T$ are correlated. A common example is when $Z$ is a non-randomized treatment variable, even after conditioning on the covariates $X$. This correlation with the error term implies that the causal effect of $Z$ on $T$ cannot be identified from the conditional distribution of $T$ on $(X,Z)$. Further, we introduce a right censoring mechanism by way of the censoring time $C$, such that only the minimum of $T$ and $C$ is observed through the follow-up time $Y=\min\{T,C\}$ and the censoring indicator $\Delta=\mathbbm{1}(T \leq C)$. We do not assume that $T$ and $C$ are independent, even after conditioning on $(X,Z)$. This possible dependence creates an additional statistical issue since the distribution of $T$ cannot be recovered from that of $(Y,\Delta)$ without further assumptions. 

\subsection{Approach} The confounding issue is tackled by the control function approach. This method uses an instrumental variable $\tilde{W}$ and the observed covariates $X$ to split $u_T$ into two parts, one which is correlated with $Z$ and one which is not. It is important that $\tilde{W}$ is a valid instrumental variable for $Z$, which means that $\tilde{W}$ is independent of $u_T$, sufficiently correlated with $Z$ and only affects $T$ through $Z$. The part of $u_T$ that is correlated with $Z$ is the control function $V$, which is assumed to be linear in $u_T$. This control function, also denoted by $g_\gamma$, is a parametric function of $(Z,X,\tilde{W})$ that is assumed to capture all confounding. Note that the parameter $\gamma$ is unknown. In this work, we propose two possible control functions for which the form follows from the relation specified between $Z$ and $(X,\tilde{W})$. Moreover, the control function allows us to estimate the causal effect of $Z$ on $T$ from the conditional distribution of $T$ on $(X,Z,V)$.
	
To allow for dependent censoring, we let $T$ and $C$, conditional on $(X,Z,V)$, follow a joint Gaussian regression model with an unspecified covariance matrix. The need for this assumption is explained in Section \ref{secmodelspec}, where it is formally introduced. Moreover, it is shown that the model is identifiable, which means that we can identify not only the causal effect of $Z$ on $T$ but also the association parameter between $T$ and $C$. This can be seen as surprising, since we only observe the minimum of $T$ and $C$ through the follow-up time $Y$ and the censoring indicator $\Delta$. In order to estimate the model parameters, a two-step estimation method is proposed. The first step estimates the parameter $\gamma$, which is required to construct the regressor $V$. Therefore, this control function $V$ can also be seen as a generated regressor. The second step uses maximum likelihood to estimate parameters of interest such as the correlation between $T$ and $C$ and the causal effect of $Z$ on $T$. Note that the second step uses the generated regressor $V$, such that a correction for the randomness coming from the first step needs to be applied to get asymptotically valid standard errors. To implement this correction, we treat the two steps as a joint generalized method of moments estimator with their moment conditions stacked in one vector. This allows us to prove consistency and asymptotic normality of the parameter estimates. Using various simulation settings, we show that the estimator demonstrates excellent finite sample performances. We illustrate the procedure by evaluating the effect of federally funded job training services on unemployment duration in the United States.

\subsection{Related literature} This paper is firstly related to the literature on dependent censoring. In the survival analysis literature, it is usually assumed that the survival time $T$ is independent of the right censoring time $C$, which is called independent censoring. However, it is easy to think of situations where this assumption is not a reasonable one to make. A common example of the independent censoring assumption being doubtful can be found in transplant studies. The survival time (time to death) is likely dependent on the censoring time (time to transplant), since selection for transplant is based on the patient's medical condition. In this case we would expect a positive dependence between $T$ and $C$, as usually the most ill patients are selected for transplant \citep{StaplinND2015Dcip}. In the literature, many methods have been proposed to handle dependent censoring. An important result comes from \citet{TsiatisA.1975ANAo}, who proved that it is impossible to identify the joint distribution of two failure times by their minimum in a fully nonparametric way. Because of this, more information about the dependence and/or marginal distributions of $T$ and $C$ is needed to identify their joint distribution. The most popular approaches are based on copulas, and \citet{ZHENGMING1995Eoms} were the first to apply this idea. Under the assumption of a fully known copula for the joint distribution of $T$ and $C$, a nonparametric estimator of the marginals was proposed. This estimator is called the copula-graphic estimator, which extends the \citet{KaplanE.L.1958NEfI} estimator to the dependent censoring case. \citet{RivestLouis-Paul2001AMAt} further investigated the copula-graphic estimator for Archimedean copulas. Note that both of these methods rely on a completely known copula. In particular, this means that the association parameter specifying the dependence between $T$ and $C$ is assumed to be known, which is often not the case in practice. The copula methods were extended to include covariates by 
\citet{BraekersRoel2005Acef},
\citet{HuangXuelin2008RSAw} and
\citet{SujicaAleksandar2018Tcei} among others. Nevertheless, these methods still rely on a fully known copula. More recently a new method was proposed by \citet{CzadoClaudia2021Dcbo}, which does not require the association parameter to be known. As a trade-off, this requires the marginals to be fully parametric for the association parameter to be identifiable. \citet{semiparderesa2020} and \citet{DeresaNegeraWakgari2020Fpmf} propose a semiparametric and parametric transformed joint regression model respectively, where the transformed variables $T$ and $C$ follow a bivariate normal distribution after adjusting for covariates. \citet{DERESA2020106879diftypecens} extends the parametric transformed joint regression model to allow for different types of censoring. The present paper relies on a similar Gaussian model as \citet{DeresaNegeraWakgari2020Fpmf}, but nevertheless differs from it as we allow for confounding. Therefore, our method can be seen as a generalization of the one proposed by \citet{DeresaNegeraWakgari2020Fpmf}. The added complication comes from the generated regressor that is introduced by the control function.

Next, the present work falls within the instrumental variable and control function literature. A confounding issue could occur due to a multitude of reasons such as noncompliance \citep{AngristJoshuaD.1996IoCE}, sample selection \citep{HeckmanJamesJ.1979SSBa}, measurement error or omitting relevant variables. The control function approach used in this work has been discussed extensively in the literature on confounding and endogeneity by \citet{LEE20071131}, \citet{Navarro2010} and \citet{WooldridgeJeffreyM2015CFMi} among others. The idea is that adding an appropriate parametric control function to the regression, which is estimated in the first stage using a valid instrument, solves the confounding issue. The advantages of this approach are that it is computationally simple, and that it can handle complicated models that are nonlinear in the confounded variable in a parsimonious manner. It is interesting to note that using the control function method creates a generated regressor problem. See \citet{pagangenreg}, \citet{oxley1993econometric} and \citet{sperlichpredvar} for an overview of possible methods and issues raised when using generated regressors. Moreover, \citet{escanciano2016identification} look at a general framework for two-step estimators with a non-parametric first step. In this work, they consider the example of a control function estimator for a binary choice model with an endogenous regressor.

Finally, the last string of research linked to this paper is that of instrumental variable methods for right censored data. We first discuss methods assuming that the censoring mechanism is independent. Some papers follow a nonparametric approach assuming that both $Z$ and $\tilde{W}$ are categorical: \citet{frandsen2015treatment}, \citet{SantAnnaPedroH.C2016PEwR} and \citet{BeyhumJad2021NIRW}. Other approaches are semiparametric, such as \citet{BijwaardGovertE2005Cfsc},
\citet{LiJialiang2015Ivah}, \citet{TchetgenTchetgenEricJ2015IVEi}, \citet{ChernozhukovVictor2015Qrwc} and \citet{beyhum2022instrumental} among others. Note that \citet{TchetgenTchetgenEricJ2015IVEi} also propose a control function approach. \citet{centorrino2021nonparametric} study nonparametric estimation with continuous regressors.
Confounding has also been discussed in a competing risks framework by
\citet{RichardsonAmy2017Nbiv}, 
\citet{ZhengCheng2017Ivwc},
\citet{MartinussenTorben2020Ivew} and
\citet{BeyhumJad2021Ania}. Research on confounding within a dependent censoring framework is sparse. Firstly, \citet{robins2000correcting} look at a correction for noncompliance and dependent censoring. However, they make the strong assumption that conditional on the treatment arm and the recorded history of six time-dependent covariates, $C$ does not further depend on $T$. It is clear that this assumption is violated if there is a variable affecting both $T$ and $C$ that is not observed. Secondly, \citet{KhanShakeeb2009Ioec} discuss an endogenously censored regression model, but they make a strong assumption (IV2, page 110 in \citet{KhanShakeeb2009Ioec}) regarding the relationship between the instruments and the covariates. An example of this assumption being violated is when the support of the natural logarithm of $C$ given $Z$ and $X$ is the whole real line, which is allowed for in our model. Finally, \citet{blanco2020bounds} look at treatment effects on duration outcomes under censoring, selection, and noncompliance. However, they derive bounds on the causal effect of $Z$ on $T$ instead of point estimates.

\subsection{Outline} This paper is structured as follows: Section \ref{secmodel} specifies the model to be studied and describes some distributions such that the expected log-likelihood can be defined. Section \ref{secidentmodel} derives the identification results and Section \ref{secestmodel} outlines the estimation procedure. Section \ref{secconsenasnorm} shows consistency and asymptotic normality for the estimator described in Section \ref{secestmodel}. Section \ref{secestasvar} describes how the asymptotic variance can be estimated. The technical details for the three theorems outlined in Section \ref{secidentenest} can be found in Appendix \ref{Theorems}. Simulation results and an empirical application regarding the impact of Job Training Partnership Act (JTPA) programs on time until employment are described in Sections \ref{secsim} and \ref{secdataappl} respectively. The code used for both of these sections can be found on \url{https://github.com/GillesCrommen}.

\section{The model}\label{secmodel}
\subsection{Model specification}\label{secmodelspec}
Let $T$ and $C$ be the logarithm of the survival time and the censoring time respectively. Contrary to the usual approach in survival analysis, the survival and censoring time are not assumed to be independent, even after conditioning on the measured covariates. Because $T$ and $C$ censor each other, only one of them is observed through the follow-up time $Y = \text{min}\{T, C\}$ and the censoring indicator $\Delta = \mathbbm{1}(T\leq C)$. The measured covariates that have an influence on both $T$ and $C$ are given by $X=(1,\Tilde{X}^\top)^\top$ and $Z$, where $\Tilde{X}$ and $Z$ are of dimension $m$ and 1 respectively. We are interested in estimating the causal effect of $Z$ on $T$, which is denoted by $\alpha_T$. A structural joint regression model can be specified as follows:
\begin{equation}\label{naivemodel}
	\left\{
	\begin{array}{ll}
		T = X^\top\beta_T + Z\alpha_T + u_T\\
		C = X^\top\beta_C + Z\alpha_C + u_C
	\end{array}
	\right.,
\end{equation}
where $Z$ is a confounded variable and $(u_T,u_C)$ are error terms representing the unobserved heterogeneity. As mentioned in the introduction, there is a confounding issue when $Z$ and $(u_T,u_C)$ are correlated. This is the case when $Z$, conditional on $X$, is non-randomized (e.g. a treatment indicator when there is non-compliance). Moreover, assuming $Z \indep (u_T,u_C)$, where $\indep$ denotes statistical independence, would lead to biased estimates of the causal effect $\alpha_T$ since $\EX[Z u_T] \neq 0$. Note that \cite{DeresaNegeraWakgari2020Fpmf} assume $Z \indep (u_T,u_C),$ with $(u_T,u_C)$ bivariate Gaussian.
	
We tackle this problem by using the control function approach, an IV method that has been discussed by \citet{LEE20071131}, \citet{Navarro2010} and \citet{WooldridgeJeffreyM2015CFMi} among others. The idea is to split the unobserved confounding variables $u_T$ and $u_C$ into two parts, one which is correlated with $Z$ and one which is not. The part of $u_T$ and $u_C$ that is correlated with $Z$ is the control function $V,$ where it is assumed that $(u_T, u_C)$ is related to $V$ through a linear model. This can be written as 
\begin{equation*}
	u_T  = \lambda_T V + \epsilon_T \qquad \text{and} \qquad u_C = \lambda_C V + \epsilon_C,
\end{equation*}
with $(\lambda_T,\lambda_C) \in \mathbb{R}^2$ and $(\epsilon_T, \epsilon_C) \indep Z$ such that only $V$ is correlated to $Z$. If $V$ were known, we could estimate the causal effect $\alpha_T$ by applying the method of \cite{DeresaNegeraWakgari2020Fpmf} to the following model:
\begin{equation}\label{model}
		\left\{
		\begin{array}{ll}
			T = X^\top{\beta_T} + Z \alpha_T + V\lambda_T  + \epsilon_T\\
			C = X^\top{\beta_C} + Z \alpha_C + V\lambda_C + \epsilon_C
		\end{array}
		\right..
\end{equation}
However, in practice, the regressor $V$ is unknown and needs to be estimated. To do this, we start by introducing a valid scalar instrument for $Z$, which is denoted by $\Tilde{W}$. By valid instrument, it is meant that: (i) $\tilde{W} \indep (u_T,u_C)$, (ii) $\tilde{W}$ is sufficiently correlated with $Z$ and (iii) $\tilde{W}$ only affects $(T,C)$ through $Z$. Further, we define $W=(X^\top,\Tilde{W})^\top$ for ease of notation. The instrument allows us to isolate the endogenous variation in $Z$. This is because, by definition, $V$ needs to be the part of $Z$ that does not depend on $W$. We achieve this by specifying a mapping $g_\gamma$, depending on a parameter $\gamma$, such that $V=g_\gamma(Z,W)$ is indeed the part of $Z$ that does not depend on $W.$ Note that the function $g$ follows from the reduced form (which is specified by the analyst), but the parameter $\gamma$ is unknown and needs to be estimated. We discuss specific examples and choices of $g$ in Section \ref{sectioncontfunc}. Moreover, it is assumed that:
\begin{enumerate}[label=(A\arabic*),left=0.25\leftmargin]
	\item
	$
	\begin{pmatrix}
		\epsilon_T\\
		\epsilon_C
	\end{pmatrix}
	\sim N_2\bigg(
	\begin{pmatrix}
		0\\
		0
	\end{pmatrix},\quad \Sigma_\epsilon = 
	\begin{pmatrix}
		{\sigma^2_T} & \rho \sigma_T\sigma_C\\
		\rho\sigma_T\sigma_C & {\sigma^2_C}
	\end{pmatrix}\bigg),$
	\\\\with $\Sigma_\epsilon$ positive definite ($\sigma_T,\sigma_C > 0$ and $\lvert\rho\rvert<1$). \label{A1}
	\item $(\epsilon_T,\epsilon_C) \indep (W,Z)$. \label{A2}
	\item The covariance matrix of $(\tilde{X}^\top,Z,V)$ is full rank and $\text{Var}(\tilde{W})>0$. \label{A3}
	\item The probabilities $\mathbb{P}(Y=T \mid W,Z)$ and $\mathbb{P}(Y=C \mid W,Z)$ are both strictly positive almost surely. \label{A4}
\end{enumerate}
While Assumptions \ref{A2}, \ref{A3} and \ref{A4} are commonly made within a survival or instrumental variable context, Assumption \ref{A1} is not. As mentioned in the introduction, \citet{TsiatisA.1975ANAo} proved that it is impossible to identify the joint distribution of two failure times by their minimum in a fully nonparametric way. Because of this result, we will need to make some assumptions regarding the dependence and/or marginal distributions of $T$ and $C$ in order to identify their joint distribution. When dependent censoring is still present after conditioning on the covariates, there are two common approaches that can be considered. The first one consists of specifying a fully known copula for the joint distribution of $T$ and $C$, while leaving the marginals unspecified (see \citet{emura2018analysis} for more details). This means that the association parameter, which describes the dependence between $T$ and $C$, is assumed to be known. As this is often not the case in practice, we opt to use a different method. At the cost of using the fully parametric model that follows from Assumption \ref{A1}, it will later be shown by Theorem \ref{thrmident} that we can actually identify the association parameter $\rho$. We deem this to be an acceptable price to pay, as there is no good way of choosing the association parameter in practice.

\subsection{Control function}\label{sectioncontfunc}
But how do we specify the control function? In the literature, different control functions are proposed which depend on a model for the relationship between $Z$ and $W$. Following \citet{wooldridge2010econometric} and \citet{Navarro2010}, we give two examples of possible control functions that will be used throughout the paper. Consider first the case where $Z$ is a continuous random variable and the relation between $Z$ and $W$ follows a linear model, that is
\begin{equation}\label{cont}
	Z=W^\top \gamma + \nu \qquad \text{with} \qquad \EX[\nu W]=0, 
\end{equation}
where $\nu$ is an unobserved error term and $\gamma\in \mathbb{R}^{m+2}$.
In this setting it is natural to set $V = g_\gamma(Z,W) = Z-W^\top \gamma$ such that $V$ is the confounded part of $Z$, that is, the part that does not depend on $W$. Another, more involved, example is when $Z$ is a binary random variable and the relation between $Z,W$ and $\nu$ is specified as 
\begin{equation}\label{bin}
	Z = \mathbbm{1}(W^\top{\gamma} - \nu  > 0) \qquad \text{with} \qquad \nu \indep W.
\end{equation}
Since we cannot directly separate $\nu$ from $Z$ and $W$, we let \begin{equation}\label{control_f}V = g_\gamma(Z,W) = Z\EX[\nu   \mid  W^\top{\gamma} > \nu ] + (1-Z)\EX[\nu   \mid   W^\top{\gamma} < \nu ].\end{equation} Then, the function $g$ is known, up to $\gamma$, when the distribution of $\nu$ is known. This specification of the control function is discussed and justified in Section 19.6.1 and Section 21.4.2 by \citet{wooldridge2010econometric}. If $\nu\sim N(0,1)$ or $\nu$ follows a standard logistic distribution, we have a probit or logit model for $Z$ respectively. Specific expressions of $g$ for the probit and logit model can be found in Appendix \ref{CF}. Moreover, when $Z$ is binary, \citet{TchetgenTchetgenEricJ2015IVEi} give another example of a possible control function: $$ V=Z - \mathbb{P}(Z = 1 \mid W). $$ Note that throughout the paper, we will use $V$ and $g_\gamma(Z,W)$ interchangeably.

\subsection{Useful distributions and definitions}\label{secdistr}
Using the assumptions that have been made so far, some conditional distributions and densities are derived. They are useful in proving the identification theorem and to define the estimator in Section \ref{secidentenest}. The expected log-likelihood function is also defined. 

For a given $\theta=(\beta_T,\alpha_T,\lambda_T,\beta_C,\alpha_C,\lambda_C,\sigma_T,\sigma_C,\rho)^\top$ and  $\gamma$, we define $F_{T \mid W,Z}(\cdot \mid w,z,\gamma;\theta)$ and $F_{C \mid W,Z}(\cdot \mid w,z,\gamma;\theta)$ as the conditional distribution function of $T$ and $C$ given $W=w =(x^\top,\tilde{w})^\top$ and $Z=z$, respectively. Thanks to Assumptions \ref{A1} and \ref{A2}, we have that:
$$
F_{T \mid W,Z}(t \mid w,z,\gamma;\theta) = \Phi\bigg(\frac{t-x^\top\beta_T - z\alpha_T - g_\gamma(z,w)\lambda_T}{\sigma_T}\bigg),$$
$$F_{C \mid W,Z}(c\mid w,z,\gamma;\theta) = \Phi\bigg(\frac{c-x^\top\beta_C - z\alpha_C - g_\gamma(z,w)\lambda_C}{\sigma_C}\bigg),
$$
with $\Phi$ the cumulative distribution function of a standard normal variable. It follows that for a given $\gamma$ and $\theta$, the conditional density functions of $T$ and $C$ given $W=w$ and $Z=z$ are, respectively:
$$
f_{T \mid W,Z}(t\mid w,z,\gamma;\theta) = {\sigma_T}^{-1} \phi\bigg(\frac{t-x^\top{\beta_T} - z\alpha_T - g_\gamma(z,w)\lambda_T}{\sigma_T}\bigg),$$
$$f_{C \mid W,Z}(c\mid w,z,\gamma;\theta) = {\sigma_C}^{-1} \phi\bigg(\frac{c-x^\top{\beta_C} - z\alpha_C - g_\gamma(z,w)\lambda_C}{\sigma_C}\bigg),
$$
where $\phi$ is the density function of a standard normal variable. For ease of notation, define $b_C = y-x^\top{\beta_C} - z\alpha_C - g_\gamma(z,w)\lambda_C$ and 
$b_T = y-x^\top{\beta_T} - z\alpha_T - g_\gamma(z,w)\lambda_T$. 
The sub-distribution function $F_{Y,\Delta \mid W,Z}(\cdot,1 \mid w,z,\gamma ;\theta)$ of $(Y,\Delta)$ given $(W,Z)$ and $(\gamma,\theta)$ can be derived as follows:
\begin{align*}
	F_{Y,\Delta \mid W,Z}(y,1\mid w,z,\gamma ;\theta) &= \mathbb{P}(Y \leq y, \Delta=1\mid W=w,Z=z) \\& = \mathbb{P}(Y \leq y, T \leq C\mid W=w,Z=z)
	\\& = \mathbb{P}(\epsilon_T\leq b_T, b_C-b_T + \epsilon_T \leq \epsilon_C).
\end{align*}
This expression is equivalent to
$$
\int_{-\infty}^{ b_T} \mathbb{P}(\epsilon_C \geq b_C - b_T + e \mid  \epsilon_T= e)f_{\epsilon_T}(e) \diff e.
$$
Since $(\epsilon_C \mid \epsilon_T=e) \sim N\big(\rho \frac{\sigma_C}{\sigma_T}e,\sigma_C^2(1-{\rho}^2)\big)$ and $\epsilon_T \sim N(0,\sigma_T^2)$, it follows that
\begin{align*}
	f_{Y,\Delta \mid W,Z}(y,1\mid w,z,\gamma ;\theta)  =\frac{1}{\sigma_T}\bigg[1-\Phi \bigg(\frac{b_C - \rho \frac{\sigma_C}{\sigma_T}b_T}{\sigma_C(1-{\rho}^2)^{\sfrac{1}{2}}}\bigg)\bigg]  \phi\bigg(\frac{b_T}{\sigma_T}\bigg).
\end{align*}
Using the same arguments, it can be shown that
\begin{align*}
	f_{Y,\Delta \mid W,Z}(y,0\mid w,z,\gamma ;\theta) =\frac{1}{\sigma_C}\bigg[1-\Phi \bigg(\frac{b_T - \rho \frac{\sigma_T}{\sigma_C}b_C}{\sigma_T(1-{\rho}^2)^{\sfrac{1}{2}}}\bigg)\bigg] \phi\bigg(\frac{b_C}{\sigma_C}\bigg).
\end{align*}
Since
$$
\mathbb{P}(Y \leq y) = \mathbb{P}(T \leq y) + \mathbb{P}(C \leq y)-\mathbb{P}(T\leq y, C\leq y),$$ we have that 
\begin{align*}
	F_{Y \mid W,Z}(y\mid w,z,\gamma ;\theta) & = \Phi\bigg(\frac{b_T}{\sigma_T}\bigg) +
	\Phi\bigg(\frac{b_C}{\sigma_C}\bigg) -\Phi\bigg(\frac{b_T}{\sigma_T}, \frac{b_C}{\sigma_C}; \rho\bigg),
\end{align*}
where $\Phi(\cdot,\cdot,\rho)$ is the distribution function of a bivariate normal distribution with covariance matrix $\begin{pmatrix}
	1 & \rho \\
	\rho & 1
\end{pmatrix}$.  Further, let
$$
S = (Y,\Delta,\Tilde{X},\Tilde{W},Z) \text{ with distribution function $G$ on } \mathcal{G} = \mathbb{R} \times \{0,1\} \times \mathbb{R}^m \times \mathbb{R} \times \mathbb{R},
$$
and
\begin{equation*}
	\ell:\mathcal{G} \times \Gamma \times \Theta \to \mathbb{R}: (s,\gamma,\theta) \mapsto \ell(s,\gamma,\theta) =  \text{log} f_{Y,\Delta \mid W,Z}(y,\delta \mid w,z,\gamma ;\theta),
\end{equation*}
where $\Theta \subset \{\theta:
(\beta_T,\alpha_T,\lambda_T,\beta_C,\alpha_C,\lambda_C) \in \mathbb{R}^{2m+6}, (\sigma_T,\sigma_C) \in \mathbb{R}^2_{>0}, \rho \in (-1,1)\}$ is the parameter space of $\theta$ and $\Gamma $ the parameter space of $\gamma$ (usually $\Gamma \subset \mathbb{R}^{m+2}$). The expected log-likelihood can be defined as follows:
$$
L(\gamma,\theta)=\EX\big[\ell(S, \gamma,\theta)\big]= \int_{\mathcal{G}} \ell(s, \gamma,\theta)\diff G(s).
$$

\section{Model identification and estimation }\label{secidentenest}
\subsection{Identification of the model}\label{secidentmodel}
We will start by showing that model \eqref{model} is identifiable in the sense that two different values of the parameter vector $(\gamma,\theta)$ result in two different distributions of $S$. Let $(\gamma^*,\theta^*)$ denote the true parameter vector. In order to prove the identifiability of the model, it will be assumed that:
\begin{enumerate}[label=(A\arabic*),resume,left=0.25\leftmargin]
	\item $\gamma^*$ is identified. \label{A5}
\end{enumerate}
Considering again the examples from Section \ref{sectioncontfunc}, when $Z$ is a continuous random variable for which \eqref{cont} holds, it is well known that the assumption that the covariance matrix of $(\tilde{X},\tilde{W})$ is full rank implies Assumption \ref{A5}. When $Z$ is a binary random variable for which \eqref{bin} holds, the assumption that the covariance matrix of $(\tilde{X},\tilde{W})$ is full rank together with a known distributional assumption on $\nu$ (e.g. $\nu \sim N(0,1)$ or $\nu \sim \text{Logistic}(0,1)$) implies Assumption \ref{A5} as shown by \citet{manskiident1988}.

\begin{theorem}\label{thrmident}
	Under Assumptions \ref{A1}-\ref{A5}, suppose that $(T_1,C_1)$ and $(T_2,C_2)$ satisfy model \eqref{model} with $(\gamma, \theta)$ and $(\gamma^*, \theta^*)$ as parameter vectors respectively. If $f_{Y_1,\Delta_1 \mid W,Z}(\cdot,k\mid w,z,\gamma ;\theta) \equiv f_{Y_2,\Delta_2 \mid W,Z}(\cdot,k\mid w,z,\gamma^* ;\theta^*)$ for almost every $(w,z)$, then
	$$
	\gamma = \gamma^* \quad \text{and} \quad \theta = \theta^*.
	$$
\end{theorem}
\noindent The proof of the theorem can be found in Appendix \ref{Theorems}. It is based on the proof of Theorem 1 by \citet{DeresaNegeraWakgari2020Fpmf}. The fact that the proposed joint regression model is identifiable can be seen as surprising, since this means that we can identify the relationship between $T$ and $C$ while only observing their minimum through the follow-up time $Y$ and the censoring indicator $\Delta$. 

\subsection{Estimation of the model parameters}\label{secestmodel}
We consider estimation when the data consist of an i.i.d. sample $\{Y_i,\Delta_i,W_i,Z_i\}_{i=1,...,n}$. Further, it is assumed that:
\begin{enumerate}[label=(A\arabic*),resume,left=0.25\leftmargin]
	\item   There exists a known function $m:(w,z,\gamma)\in\mathbb{R}^{m+2}\times \mathbb{R}\times \Gamma \mapsto m(w,z,\gamma)$ twice continuously differentiable with respect to $\gamma$ such that the estimator \begin{equation}\label{gammaest}
		\hat{\gamma} \in \argmax_{\gamma \in \Gamma } n^{-1} \displaystyle \sum_{i=1}^{n} m(W_i,Z_i,\gamma),     
	\end{equation}
	is consistent for the true parameter $\gamma^*$. \label{A6}
\end{enumerate}
Using the first-order conditions of program \eqref{gammaest}, we obtain that $n^{-1} \sum_{i=1}^{n} \nabla_\gamma m(W_i,Z_i,\hat{\gamma})=0$. Hence, Assumption \ref{A6} implies that we possess a consistent $Z$-estimator of $\gamma$. The theory on $M$-estimators \citep{newey1994large} allows us to find sufficient conditions for the assumption that $\hat\gamma$ is consistent. Assumption \ref{A6} will hold when (i) the true parameter $\gamma^*$ belongs to the interior of $\Gamma$, which is compact, (ii) $\mathcal{L}(\gamma)=\EX\big[m(W,Z,\gamma)\big]$ is continuous and uniquely maximized at $\gamma^*$ and (iii) $\hat{\mathcal{L}}(\gamma)=n^{-1}\sum_{i=1}^{n} m(W_i,Z_i,\gamma)$ converges uniformly (in $\gamma\in \Gamma$) in probability to $\mathcal{L}(\gamma)$. In the case where $\hat{\mathcal{L}}(\cdot)$ is concave, (i) can be weakened to $\gamma^*$ being an element of the interior of a convex set $\Gamma$, while (iii) is only required to hold pointwise rather than uniformly. Returning again to the examples given in Section \ref{sectioncontfunc}, when $Z$ is a continuous random variable for which \eqref{cont} holds, it is well known that ordinary least squares is an extremum estimation method that consistently estimates $\gamma$ under the assumption that the covariance matrix of $(\tilde{X},\tilde{W})$ is full rank. In this case, we can define $m(W,Z,\gamma) = - (Z - W^\top \gamma)^2$. When $Z$ is a binary random variable for which \eqref{bin} holds and the distribution of $\nu$ is known, maximum likelihood estimation can be used to consistently estimate $\gamma$ under weak regularity conditions that can be found in \citet{AldrichJohnH1991Lpla}. In this case, we can define $m(W,Z,\gamma)= Z\log \mathbb{P}(W^\top{\gamma} > \nu )+{(1-Z)}\log \mathbb{P}(W^\top{\gamma} < \nu )$. After obtaining $\hat{\gamma}$ from \eqref{gammaest}, the parameters from model \eqref{model} can be estimated using maximum likelihood with the estimates given by the second-step estimator:
\begin{equation}\label{thetaest}
	\hat{\theta}=(\hat{\beta}_T,\hat{\alpha}_T,\hat{\lambda}_T,\hat{\beta}_C,\hat{\alpha}_C,\hat{\lambda}_C,\hat{\sigma}_T,\hat{\sigma}_C,\hat{\rho})=\argmax_{\theta \in \Theta} \hat{L}(\hat{\gamma}, \theta),
\end{equation}
with $\Theta$ the parameter space as defined before and
\begin{align*}
	&\hat{L}(\hat{\gamma},\theta) =\frac{1}{n} \displaystyle \sum_{i=1}^{n} \text{ log} f_{Y,\Delta \mid W,Z}(Y_i,\Delta_i \mid W_i,Z_i,\hat{\gamma};\theta)
	\\&=\frac{1}{n}\displaystyle \sum_{i=1}^{n}\Bigg\{{\Delta_i}\Bigg(-\text{log}(\sigma_T)+\text{log}\bigg[1-\Phi \bigg(\frac{b_{C_i} - \rho \frac{\sigma_C}{\sigma_T}b_{T_i}}{\sigma_C(1-\rho^2)^{\sfrac{1}{2}}}\bigg)\bigg]+\text{log}\bigg[\phi\bigg(\frac{b_{T_i}}{\sigma_T}\bigg)\bigg]\Bigg)\\& \quad +(1-\Delta_i)\Bigg(-\text{log}(\sigma_C)+\text{log}\bigg[1-\Phi \bigg(\frac{b_{T_i} - \rho \frac{\sigma_T}{\sigma_C}b_{C_i}}{\sigma_T(1-\rho^2)^{\sfrac{1}{2}}}\bigg)\bigg]+\text{log}\bigg[\phi\bigg(\frac{b_{C_i}}{\sigma_C}\bigg)\bigg]\Bigg)\Bigg\},
\end{align*}
with 
$$b_{C_i}= Y_i  -  X_i^\top \beta_C -Z_i \alpha_C-g_{\hat{\gamma}}(W_i,Z_i)\lambda_C,$$
$$b_{T_i}= Y_i -  X_i^\top\beta_T-Z_i\alpha_T-g_{\hat{\gamma}}(W_i,Z_i)\lambda_T.$$

\subsection{Consistency and asymptotic normality}\label{secconsenasnorm}
In this section, it will be shown that the parameter estimates $\hat{\theta}$, as defined in \eqref{thetaest}, are consistent and asymptotically normal. Theorems \ref{thrmcons} and \ref{thrmasn} show consistency and asymptotic normality respectively. The proofs can be found in Appendix \ref{Theorems}. We start by providing some definitions and assumptions that will be useful in stating these theorems. Let
\begin{alignat*}{2}
	&h_\ell(S,\gamma^{*},\theta^{*}) = \nabla_\theta \ell(S, \gamma^{*}, \theta^{*}),  && \quad H_\theta = \EX\big[\nabla_\theta h_\ell(S, \gamma^{*},\theta^{*})\big],  \\ &  h_m(W,Z,\gamma^*) = \nabla_\gamma m(W,Z,\gamma^{*}), && \quad  H_\gamma = \EX\big[\nabla_\gamma h_\ell(S,\gamma^{*},\theta^{*})\big], \\ &   M = \EX\big[\nabla_\gamma h_m(W,Z,\gamma^{*})\big], && \quad  \Psi = -M^{-1}h_m(W,Z,\gamma^{*}), \\ & \Tilde{h}(S,\gamma^*,\theta^*) = \big(h_m(W,Z,\gamma^*)^\top, h_\ell(S,\gamma^*,\theta^*)^\top  \big)^\top, && \quad  H = \EX \Big[\nabla_{\gamma,\theta}\Tilde{h}(S,\gamma^*,\theta^*) \Big].
\end{alignat*}
The following assumptions will be used in the proofs of Theorems \ref{thrmcons} and \ref{thrmasn}:
\begin{enumerate}[label=(A\arabic*),resume,left=0.25\leftmargin]
	\item The parameter space $\Theta$ is compact and $\theta^* $ belongs to the interior of $\Theta$. \label{A7}
	\item There exists a function $\mathcal{D}(s)$ integrable with respect to $G$ and a compact neighborhood $\mathcal{N}_{\gamma} \subseteq \Gamma$ of $\gamma^*$ such that $\lvert\ell(s, \gamma, \theta)\rvert \leq \mathcal{D}(s)$ for all $\gamma\in\mathcal{N}_\gamma$ and $ \theta \in \Theta$. \label{A8}
	\item {$\EX\Big[\lVert\Tilde{h}(S,\gamma^*,\theta^*)\rVert^2\Big] < \infty$ and $\EX\Bigg[\sup\limits_{(\gamma,\theta) \in \mathcal{N}_{\gamma,\theta}}\lVert \nabla_{\gamma,\theta}  \Tilde{h}(S,\gamma,\theta)\rVert\Bigg] < \infty, \text{ with } \mathcal{N}_{\gamma,\theta}$ a neighborhood of $(\gamma^*,\theta^*)$ in $\Gamma\times \Theta$.}\label{A9}
	\item $H^\top H$ is nonsingular.\label{A10}
\end{enumerate}
Note that $\lVert \cdot \rVert$ represents the Euclidean norm. Assumption \ref{A8} is necessary to show the consistency and asymptotic normality of the parameter estimates. Sufficient conditions for this assumption are that the support of $S$ is bounded, $\Gamma$ being compact and Assumption \ref{A7}. Assumptions \ref{A7}, \ref{A9} and \ref{A10} are regularity conditions that are commonly made in a maximum likelihood context. We have the following consistency theorem.

\begin{theorem}\label{thrmcons}
	Under Assumptions \ref{A1}-\ref{A8}, suppose that $\hat{\gamma}$ and $\hat{\theta}$ are parameter estimates as described in \eqref{gammaest} and \eqref{thetaest} respectively, then
	$$\hat{\gamma}\xrightarrow{\text{ p }}\gamma^* \quad \text{and} \quad \hat{\theta}\xrightarrow{\text{ p }}\theta^*.$$
\end{theorem}

The challenge in proving this theorem comes from the fact that we are using a two-step estimation method, meaning that the results from the first step are used in the second step. To ensure consistency of $\hat{\theta}$, in the proofs, we show uniform convergence (in $\theta \in \Theta$) in probability of the empirical likelihood function $ \hat{L}(\hat{\gamma},\theta) $ in \eqref{thetaest} to the true likelihood of the model at $\gamma^*$. We also have the following asymptotic normality result:

\begin{theorem}\label{thrmasn}
	Under Assumptions \ref{A1}-\ref{A10}, suppose that $\hat{\theta}$ is a parameter estimate as described in \eqref{thetaest}, then
	$$
	\sqrt{n}(\hat{\theta}-\theta^{*}) \xrightarrow{\text{ d }} N(0,\Sigma_\theta), \quad$$ with $$\quad \Sigma_\theta = H^{-1}_\theta \EX\big[\{h_\ell(S,\gamma^{*}, \theta^{*}) + H_\gamma \Psi\}\{h_\ell(S,\gamma^{*}, \theta^{*}) + H_\gamma \Psi\}^\top\big]\big({H^{-1}_\theta}\big)^\top.
	$$
\end{theorem}

The difficulty in proving this theorem is related to the fact that the randomness coming from the first step inflates the asymptotic variance of $\hat{\theta}$. Hence, ignoring the first step would lead to inconsistent standard errors and confidence intervals that are not asymptotically valid. To obtain correct standard errors, we treat the two steps as a joint generalized method of moments (GMM) estimator with their moment conditions stacked in one vector \citep{newey1994large}. Indeed, given that $\hat{\gamma}$ and $\hat{\theta}$ are consistent by Theorem \ref{thrmcons}, they are the unique solutions to the first order conditions of their respective objective functions in a neighborhood of $\gamma^*$ and $\theta^*$ (with probability going to $1$). Therefore, the two-step estimator is asymptotically equivalent to the GMM estimator corresponding to the following moments: 
$$\EX[h_m(W,Z,\gamma)]=0 \text{ and } \EX[h_\ell(S, \gamma,\theta)]=0,$$ for the first and second step respectively. As a last remark, if we were to remove the correction $H_\gamma \Psi$ for the first step, the covariance matrix simplifies to the inverse of Fisher's information matrix (assuming the model is correctly specified).

\subsection{Estimation of the asymptotic variance}\label{secestasvar}
Using the result from Theorem \ref{thrmasn}, we can construct a consistent estimator $\hat{\Sigma}_\theta$ for the covariance matrix of the parameters in $\theta$ in the following way:
$$
\hat{\Sigma}_\theta= \hat{H}^{-1}_\theta \Bigg[n^{-1}{\displaystyle \sum_{i=1}^{n}}\{h_\ell(S_i,\hat{\gamma},\hat{\theta}) + \hat{H}_\gamma \hat{\Psi}_i\}\{h_\ell(S_i,\hat{\gamma},\hat{\theta}) + \hat{H}_\gamma \hat{\Psi}_i\}^\top\Bigg]{\big(\hat{H}^{{-1}}_\theta\big)}^\top,
$$
where $S_i=\big(Y_i,\Delta_i,\Tilde{X}_i,\Tilde{W}_i,Z_i\big)$ and
\begin{alignat*}{3}
	& h_\ell(S_i,\hat{\gamma},\hat{\theta}) = \nabla_\theta \ell(S_i, \hat{\gamma},\hat{\theta}), && \quad h_m(W_i,Z_i,\hat{\gamma}) = \nabla_\gamma m(W_i,Z_i,\hat{\gamma}),  && \quad \hat{H}_\theta = n^{-1}{\displaystyle \sum_{i=1}^{n}} \nabla_\theta h_\ell(S_i,\hat{\gamma},\hat{\theta}),  \\ & \hat{H}_\gamma = n^{-1}{\displaystyle \sum_{i=1}^{n}} \nabla_\gamma h_\ell(S_i,\hat{\gamma},\hat{\theta}), && \quad \hat{M} = n^{-1}{\displaystyle \sum_{i=1}^{n}} \nabla_\gamma h_m(W_i,Z_i,\hat{\gamma}), && \quad \hat{\Psi}_i = -\hat{M}^{-1}h_m(W_i,Z_i,\hat{\gamma}).
\end{alignat*}
Thanks to the asymptotic normality and the consistent estimator for the variance of the estimators, confidence intervals can easily be constructed. Note that since $\sigma_T, \sigma_C > 0$ and $\rho \in (-1,1)$, their confidence intervals will be constructed using a logarithm and a Fisher's z-transformation respectively. These transformations project the estimates on the real line, after which the delta method can be used to obtain their standard errors. The confidence intervals can then be constructed and transformed back to the original scale. This procedure makes sure that our confidence intervals are reasonable (e.g. no negative values for the confidence limits of the standard deviation estimates).
Also note that instead of calculating $h_\ell(S_i,\hat{\gamma},\hat{\theta}), \hat{H}_\theta$ and $\hat{H}_\gamma$ using their analytical expressions, they are approximated. This is due to the complexity of these expressions and the amount of them that would have to be derived. For example, $\hat{H}_\theta$ is already a $(2m+9) \times (2m+9)$ matrix of derivatives where $m$ is the dimension of $\Tilde{X}$. The calculation of these approximations is done by making use of Richardson's extrapolation \citep{richardson_1911}, resulting in more accurate estimates. A general description of the method to approximate the Jacobian matrix can be given as repeated calculations of the central difference approximation of the first derivative with respect to each component of $\theta$, using a successively smaller step size. Richardson's extrapolation uses this information to estimate what happens when the step size goes to zero. A similar description can be given for the approximation of the Hessian matrices. Note that these calculations can be quite time consuming depending on the required level of accuracy.

\section{Simulation study}\label{secsim}

In this section, a simulation study is performed to investigate the finite sample performance of the proposed two-step estimator. We consider the four combinations of the cases where $Z$ and $\tilde{W}$ are continuous or binary random variables. It is assumed that when $Z$ is binary, it follows a logit model. The proposed estimator is compared to three other estimators: one which does not account for the confounding issue, one which assumes $T$ and $C$ are independent and one which uses the proposed method but treats $V$ as observed. The parameters are estimated for samples of 250, 500 and 1000 observations. The first step of the data generating process is as follows:
$$
\begin{pmatrix}
	\epsilon_T\\
	\epsilon_C
\end{pmatrix}
\sim N_2\bigg(
\begin{pmatrix}
	0\\
	0
\end{pmatrix},\quad \Sigma = 
\begin{pmatrix}
	{1.1}^2 & 0.75\cdot 1.1 \cdot 1.4\\
	0.75\cdot 1.1 \cdot 1.4 & {1.4}^2
\end{pmatrix}\bigg), \quad \Tilde{X} \sim N(0, 1).
$$
We have 4 different designs depending on whether $Z$ and $\tilde{W}$ are assumed to be a continuous or binary random variable:
\begin{table}[H]
	\centering
	\renewcommand{\arraystretch}{1.25}
	\begin{tabular}{|c|c|c|}
		\hline
		Design 1& $\Tilde{W} \sim U[0, 2], \quad \nu \sim N(0,2)$ & \multirow{ 2}{*}{ $Z = W^{\top}\gamma + \nu,$} \\ \cline{1-2}
		Design 2& $\Tilde{W} \sim \text{Bernoulli}(0.5), \quad \nu \sim N(0,2)$&  \\   \hline
		Design 3& $\Tilde{W} \sim U[0, 2], \quad \nu \sim \text{Logistic}(0,1)$& \multirow{ 2}{*}{ $Z =\mathbbm{1}(W^\top{\gamma} - \nu > 0),$}  \\   \cline{1-2}
		Design 4& $\Tilde{W} \sim \text{Bernoulli}(0.5), \quad \nu \sim \text{Logistic}(0,1)$&  \\   \hline
	\end{tabular}
	\renewcommand{\arraystretch}{1}
	\label{designs1}
\end{table}
\noindent with $W = (1,\tilde{X},\tilde{W})^\top$ and $\gamma = (-1, 0.6, 2.3)^{\top}$. It is also assumed that $$(\epsilon_T,\epsilon_C) \indep (\tilde{X},\tilde{W}, \nu), \qquad \tilde{W} \indep (\tilde{X}, \nu) \quad \text{and} \quad \tilde{X} \indep \nu.$$ 
From this, we can construct:
\begin{table}[H]
	\centering
	\renewcommand{\arraystretch}{1.5}
	\begin{tabular}{|c|c|}
		\hline
		Design 1& \multirow{ 2}{*}{ $V = Z - W^\top{\gamma}.$} \\ [3pt]\cline{1-1}
		Design 2 &   \\ [3pt] \hline
		Design 3 & $ V = (1-Z)\Big[\big(1+\exp\{W^\top{\gamma}\}\big)\log\big(1+\exp\{W^\top{\gamma}\}\big)-W^\top{\gamma}\exp\{W^\top{\gamma}\}\Big]$  \\ [3pt]  \cline{1-1}
		Design 4 & $-Z\Big[\big(1+\exp\{-W^\top{\gamma}\}\big)\log\big(1+\exp\{-W^\top{\gamma}\}\big)+W^\top{\gamma}\exp\{-W^\top{\gamma}\}\Big].$  \\ [3pt]  \hline
	\end{tabular}
	\renewcommand{\arraystretch}{1}
	\label{designs2}
\end{table}
\noindent Finally, $T$ and $C$ can be constructed for each design in the following way:
$$
\left\{
\begin{array}{ll}
	T = \beta_{T,0} + \Tilde{X}\beta_{T,1} + Z \alpha_T + V \lambda_T  + \epsilon_T\\
	C = \beta_{C,0} + \Tilde{X}\beta_{C,1} + Z \alpha_C + V\lambda_C + \epsilon_C
\end{array}
\right.,
$$
where
$(\beta_{T,0}, \beta_{T,1}, \alpha_T, \lambda_T)  = (2.5, 2.6, 1.8, 2) \text{ and } (\beta_{C,0}, \beta_{C,1}, \alpha_C, \lambda_C) =  (2.8, 1.9, 1.5, 1.2).$
It follows that $Y$ = min\{$T,C$\} and $\Delta = \mathbbm{1}(T\leq C)$. 

This data generating process was repeated 2500 times for the four possible designs. The parameter values were chosen such that there is between 45\% and 50\% censoring for each design. For each sample size, there are four different estimators. The first, which we call the naive estimator, ignores the confounding issue and therefore does not include $V$ in the model (no estimates for $\lambda_T$ and $\lambda_C$). The second, which we call the independent estimator, assumes that $T$ and $C$ are independent from each other (no estimates for $\rho$ as it is assumed to be zero). The third, which we call the oracle estimator, uses the control function approach to handle the confounding issue but treats $V$ as if it were observed. The fourth and last, which we call the two-step estimator, uses the two-step estimation method proposed in this article. This means that $V$ is constructed using $\hat{\gamma}$ from the first step. The estimation is performed in \textit{R} and uses the package \textit{nloptr} to maximize certain functions and the package \textit{numDeriv} for computing the necessary Hessian and Jacobian matrices. The package \textit{MASS} is used to generate the bivariate normal variables.

For each estimator, the bias of each parameter estimate is given together with the empirical standard deviation (ESD) and the root mean squared error (RMSE). Note that, as the bias decreases, these last 2 statistics should converge to the same value. To better explain how these statistics are calculated, we give the formulas for $\alpha_T$ as an example. Let $N$ represent the total amount of simulations with $j=1,...,N$ and $(\hat{\alpha}_T)_j$ the estimate of $\alpha_T$ for the $j$'th simulation. The ESD and RMSE for $\alpha_T$ are given as follows:
\begin{align*}
	&\text{ESD} =\sqrt{(N-1)^{-1}\sum_{j=1}^{N}\big[(\hat{\alpha}_T)_j-\Bar{\alpha}_T\big]^2},\quad \text{with } \Bar{\alpha}_T = N^{-1}\sum_{j=1}^{N}(\hat{\alpha}_T)_j.
	\\&\text{RMSE} = \sqrt{N^{-1}\sum_{j=1}^{N}\big[(\hat{\alpha}_T)_j-\alpha_T\big]^2}, \quad \text{with } \alpha_T \text{ the real parameter value}.
\end{align*}
Lastly, the coverage rate (CR) shows in which percentage of the simulations the real parameter value is included in the estimated 95\% confidence interval.

\begin{table}[H]
	\centering
	\caption{Estimation results for design 4 with 46\% censoring and 2500 simulations. Given are the bias, the empirical standard deviation (ESD), the root mean squared error (RMSE) and the confidence rate (CR).}
	\resizebox{\textwidth}{!}{
		\begin{tabular}{|c|rrrr|rrrr|rrrr|}
			\hline
			\multicolumn{5}{|c|}{$n=250$} & \multicolumn{4}{c|}{$n=500$} & \multicolumn{4}{c|}{$n=1000$}\\
			\hline
			\multicolumn{13}{|c|}{naive estimator} \\
			\hline
			& Bias & ESD & RMSE & CR & Bias & ESD & RMSE & CR & Bias & ESD & RMSE & CR \\ 
			\hline
			$\beta_{T,0}$ & 2.450 & 0.536  & 2.508 & 0.008 & 2.478 & 0.358 & 2.504 & 0.000 & 2.481 & 0.254  & 2.494 & 0.000 \\  
			$\beta_{T,1}$ & 0.525 & 0.254  & 0.584 & 0.362 & 0.533 & 0.177 & 0.561 & 0.108 & 0.535 & 0.124  & 0.549 & 0.004 \\ 
			$\alpha_T$ & -4.526 & 0.506  & 4.554 & 0.001 & -4.566 & 0.311 & 4.577 & 0.000 & -4.571 & 0.221  & 4.576 & 0.000\\ 
			$\beta_{C,0}$ & 1.456 & 0.256  & 1.479 & 0.009 & 1.474 & 0.152 & 1.481 & 0.000 & 1.472 & 0.108  & 1.476 & 0.000\\ 
			$\beta_{C,1}$ & 0.238 & 0.258  & 0.351 & 0.807 & 0.229 & 0.182 & 0.293 & 0.720 & 0.228 & 0.133  & 0.263 & 0.548 \\ 
			$\alpha_C$ & -2.615 & 0.461  & 2.655 & 0.005 & -2.627 & 0.322 & 2.646 & 0.000 & -2.618 & 0.231  & 2.628 & 0.000 \\ 
			$\sigma_T$ & 0.582 & 0.114  & 0.593 & 0.001 & 0.581 & 0.079 & 0.587 & 0.000 & 0.581 & 0.057 & 0.584 & 0.000 \\ 
			$\sigma_C$ & 0.221 & 0.105  & 0.245 & 0.320 & 0.223 & 0.067 & 0.233 & 0.039 & 0.223 & 0.048  & 0.228 & 0.001 \\ 
			$\rho$ & -0.057 & 0.263  & 0.269 & 0.930 & -0.042 & 0.184 & 0.189 & 0.936 & -0.031 & 0.129 & 0.132 & 0.928 \\ 
			\hline
			\multicolumn{13}{|c|}{independent estimator} \\
			\hline
			$\beta_{T,0}$ & 0.454 & 0.491 & 0.668 & 0.761 & 0.450 & 0.352 & 0.572 & 0.674 & 0.465 & 0.245  & 0.526 & 0.480  \\ 
			$\beta_{T,1}$ & 0.229 & 0.221  & 0.318 & 0.781 & 0.224 & 0.154 & 0.272 & 0.663 & 0.225 & 0.111  & 0.250 & 0.450 \\
			$\alpha_T$ & 0.036 & 0.838  & 0.839 & 0.939 & 0.034 & 0.600 & 0.601 & 0.945 & 0.009 & 0.410  & 0.410 & 0.947 \\ 
			$\lambda_T$ & 0.223 & 0.326  & 0.395 & 0.967 & 0.221 & 0.233 & 0.321 & 0.904 & 0.213 & 0.161  & 0.267 & 0.782\\
			$\beta_{C,0}$ & 0.578 & 0.369  & 0.686 & 0.593 & 0.567 & 0.260 & 0.624 & 0.385 & 0.576 & 0.179  & 0.603 & 0.124\\ 
			$\beta_{C,1}$ & -0.291 & 0.174  & 0.339 & 0.610 & -0.289 & 0.122  & 0.313 & 0.317 & -0.285 & 0.085  & 0.297 & 0.072 \\ 
			$\alpha_C$ & 0.020 & 0.616  & 0.616 & 0.954 & 0.029 & 0.437 & 0.438 & 0.952 & 0.015 & 0.306 & 0.306 & 0.951\\ 
			$\lambda_C$ & -0.251 & 0.243  & 0.349 & 0.756 & -0.242 & 0.171 & 0.296 & 0.654 & -0.248 & 0.118 & 0.275 & 0.439\\ 
			$\sigma_T$ & 0.048 & 0.070  & 0.085 & 0.888 & 0.053 & 0.049 & 0.072 & 0.801 & 0.056 & 0.036 & 0.067 & 0.618 \\ 
			$\sigma_C$ & 0.191 & 0.107 & 0.219 & 0.515  & 0.202 & 0.073 & 0.214 & 0.165 & 0.207 & 0.053 & 0.214 & 0.011 \\  
			\hline
			\multicolumn{13}{|c|}{oracle estimator} \\
			\hline
			$\beta_{T,0}$ & -0.001 & 0.223 & 0.223 & 0.952 & -0.001 & 0.154  & 0.154 & 0.952 & 0.003 & 0.109  & 0.109 & 0.949 \\ 
			$\beta_{T,1}$ & -0.003 & 0.115  & 0.115 & 0.951 & -0.001 & 0.082  & 0.082 & 0.947 & -0.000 & 0.056  & 0.056 & 0.955 \\ 
			$\alpha_T$ & -0.007 & 0.322 & 0.322 & 0.942 & -0.003 & 0.215  & 0.215 & 0.957 & -0.004 & 0.158  & 0.158 & 0.945\\  
			$\lambda_T$ & -0.003 & 0.151  & 0.151 & 0.943 & -0.002 & 0.101  & 0.101 & 0.960 & -0.001 & 0.074  & 0.074 & 0.944 \\ 
			$\beta_{C,0}$ & -0.004 & 0.298  & 0.298 & 0.940 & -0.004 & 0.209  & 0.209 & 0.948 & 0.005 & 0.151  & 0.151 & 0.952\\ 
			$\beta_{C,1}$ & 0.006 & 0.151 & 0.151 & 0.937 & -0.000 & 0.104  & 0.104 & 0.944 & 0.000 & 0.074  & 0.074 & 0.944\\ 
			$\alpha_C$ & -0.003 & 0.405  & 0.405 & 0.948 & 0.003 & 0.282  & 0.282 & 0.954 & -0.007 & 0.204  & 0.204 & 0.946\\ 
			$\lambda_C$ & 0.002 & 0.186 & 0.186 & 0.942 & 0.003 & 0.132  & 0.132 & 0.946 & -0.004 & 0.094  & 0.094 & 0.946\\ 
			$\sigma_T$ & -0.007 & 0.058 & 0.059 & 0.945 & -0.004 & 0.041  & 0.041 & 0.950 & -0.002 & 0.030  & 0.030 & 0.943 \\ 
			$\sigma_C$ & -0.018 & 0.092  & 0.094 & 0.933 & -0.009 & 0.064  & 0.065 & 0.951 & -0.004 & 0.046  & 0.046 & 0.944\\ 
			$\rho$ & -0.009 & 0.142 & 0.143 & 0.970 & -0.006 & 0.101  & 0.101 & 0.953 & -0.004 & 0.069 & 0.069 & 0.955 \\ 
			\hline
			\multicolumn{13}{|c|}{two-step estimator} \\
			\hline
			$\beta_{T,0}$ & -0.043 & 0.456  & 0.458 & 0.954 & -0.023 & 0.311  & 0.312 & 0.956 & -0.007 & 0.214  & 0.214 & 0.952\\ 
			$\beta_{T,1}$ & -0.021 & 0.227 & 0.228 & 0.952 & -0.009 & 0.157  & 0.157 & 0.952 & -0.005 & 0.109  & 0.109 & 0.952 \\ 
			$\alpha_T$ & 0.061 & 0.782  & 0.784 & 0.944 & 0.036 & 0.525  & 0.526 & 0.952 & 0.012 & 0.365  & 0.365 & 0.951 \\  
			$\lambda_T$ & 0.018 & 0.317 & 0.318 & 0.942 & 0.011 & 0.213  & 0.213 & 0.952 & 0.003 & 0.151  & 0.151 & 0.941\\ 
			$\beta_{C,0}$ & -0.023 & 0.375  & 0.375 & 0.958 & -0.015 & 0.261  & 0.261 & 0.961 & 0.000 & 0.184 & 0.184 & 0.951 \\ 
			$\beta_{C,1}$ & -0.004 & 0.184  & 0.184 & 0.942 & -0.004 & 0.126  & 0.126 & 0.948 & -0.002 & 0.088  & 0.088 & 0.953\\
			$\alpha_C$ & 0.032 & 0.595 & 0.596 & 0.947 & 0.025 & 0.403 & 0.403 & 0.954 & 0.002 & 0.285 & 0.285 & 0.952\\ 
			$\lambda_C$ & 0.012 & 0.250 & 0.250 & 0.956 & 0.010 & 0.173 & 0.174 & 0.957 & -0.001 & 0.123 & 0.123 & 0.952 \\ 
			$\sigma_T$ & -0.004 & 0.058 & 0.059 & 0.946 & -0.003 & 0.041 & 0.041 & 0.944 & -0.001 & 0.030  & 0.030 & 0.940\\ 
			$\sigma_C$ & -0.018 & 0.093 & 0.095 & 0.932 & -0.009 & 0.065 & 0.065 & 0.945  & -0.004 & 0.047 & 0.047 & 0.946 \\
			$\rho$ & -0.009 & 0.147 & 0.147 & 0.970 & -0.006 & 0.104 & 0.104 & 0.955 & -0.004 & 0.072 & 0.072 & 0.954 \\ 
			\hline
		\end{tabular}
	}
	\label{table22}
\end{table}

Table \ref{table22} shows the results for design 4, meaning that both $Z$ and $W$ are binary. The results show a very noticeable bias for the naive estimator for almost each parameter estimate. Note that this bias remains the same as the sample size increases. The table also shows that the estimated standard errors are not asymptotically valid as the CR is inconsistent and does not converge to the expected 95\%. We find the same results when looking at the independent estimator but to a lesser extent. The bias is lower compared to the naive estimator but is still noticeable. A difference with the naive estimator is that there is a lot less bias for $\alpha_T$ and $\alpha_C$. This is to be expected as the independent estimator does take the confounding issue into account. However, the estimated standard errors for $\alpha_T$ and $\alpha_C$ are not asymptotically valid as the CR is inconsistent. It may seem that the CR converge to 95\%, but Tables 4 and 6 show that this is not always the case. As with the naive estimator, the bias does not necessarily decrease when the sample size increases. The bias for the proposed two-step estimation method is very close to 0 and is clearly an improvement over the naive and independent estimator. It is also very close to  that of the oracle estimator, which treats $V$ as observed. The bias decreases when the sample size increases and the ESD and RMSE converge to the same value, which also decreases as the sample size increases. The CR is around 95\%, meaning that we have asymptotically valid standard errors and confidence intervals. From this, it is clear that the two-step estimator performs well, even for small sample sizes. The results for the other designs are similar and can be found in Appendix \ref{Tables}.

\section{Data application}\label{secdataappl}

In this section, we apply the outlined methodology to estimate the effect of Job Training Partnership Act (JTPA) services on time until employment. The data come from a large-scale randomized experiment known as the National JTPA Study and have been analyzed extensively by \citet{bloom1997benefits}, \citet{abadie2002instrumental} and \citet{frandsen2015treatment} among others. The data and problem investigated is the same as in \citet{frandsen2015treatment}, but the method used nevertheless differs as we allow for dependent censoring. Later in this section, we give our reasoning as to why there could be dependent censoring present in the data

This study was performed to evaluate the effectiveness of more than 600 federally funded services, established by the Job Training Partnership Act of 1982, that were intended to increase the employability of eligible adults and out-of-school youths. These services included classroom training, on-the-job training and job search assistance. The JTPA started to fund these programs in October of 1983 and continued funding up until the late 1990's. Between 1987 and 1989, a little over 20,000 adults and out-of-school youths who applied for JTPA were randomly assigned to be in either a treatment or a control group. Treatment group members were eligible to receive JTPA services, while control group members were not eligible for 18 months. However, due to local program staff not always following the randomization rules closely, about 3\% of the control group members were able to participate in JTPA services. It is important to note that we are not comparing JTPA services to no services but rather JTPA services versus no and other services, since control group members were still eligible for non-JTPA training. Between 12 and 36 months after randomization, with an average of 21 months, the participants were surveyed by data collection officers. Next, a subset of 5,468 subjects participated in a second follow-up survey, which focused on the period between the two surveys. The second survey took place between 23 and 48 months after randomization. See Figure \ref{fig1} for a graphical representation of the interview process.
\newpage
\begin{figure}[H]
	\centering
	\scalebox{1}{
	\begin{tikzpicture}[node distance=2cm]
		\node (dec1) [decision, align=center] {individual $i$ has \\ interview 1 at time $I_{1,i}$};
		\node (pro1) [process, right of=dec1, yshift=3cm, align=center] {invited to \\ 2nd interview};
		\node (pro2) [process, right of=dec1, yshift=-3cm, align=center] {not invited to \\ 2nd interview};
		\node (pro3) [observation, right of=pro2, xshift=3cm] {$T < I_{1,i}$ or $C = I_{1,i}$};
		\node (dec2) [decision, right of=pro1, xshift=3cm, align=center] {individual $i$ has \\ interview 2 at time $I_{2,i}$};
		\node (pro4) [process, below of=dec2, yshift=-1cm, align=center] {no response to \\ invitation};
		\node (obs1) [observation, right of=dec2, xshift=3cm] {$ T < I_{2,i}$ or $C = I_{2,i}$};
		\draw [arrow] (dec1) -- (pro1);
		\draw [arrow] (dec1) -- (pro2);
		\draw [arrow] (pro2) -- (pro3);
		\draw [arrow] (pro1) -- (pro4);
		\draw [arrow] (pro4) -- (pro3);
		\draw [arrow] (pro1) -- (dec2);
		\draw [arrow] (dec2) -- (obs1);
	\end{tikzpicture}}
	\captionof{figure}{JTPA interview process for individual $i$.}
	\label{fig1}
\end{figure}
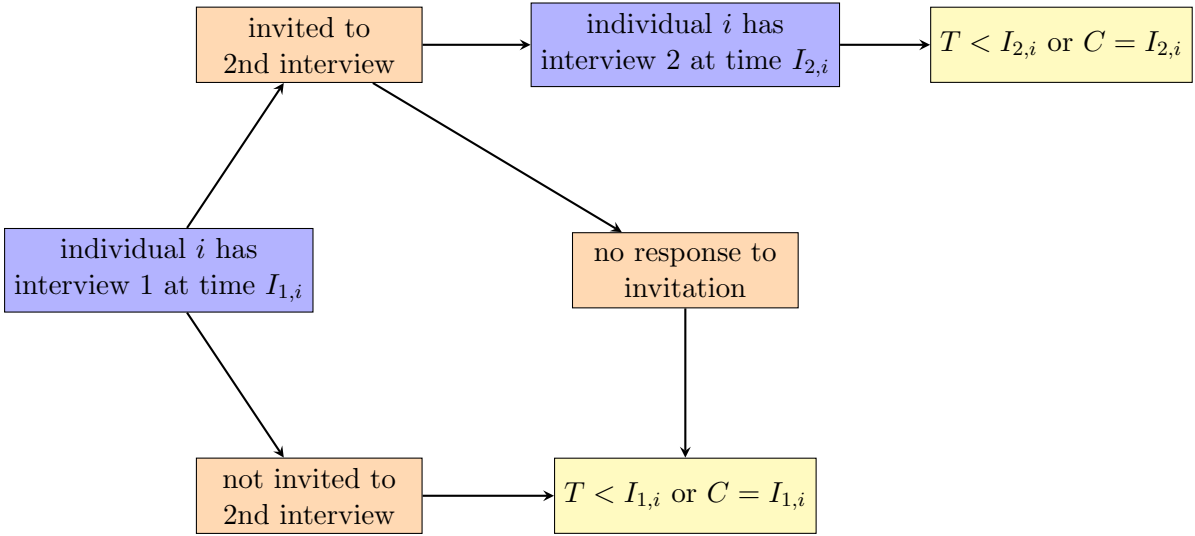 
In this application, we will focus our attention on the effect of JTPA programs on the sample of 1,298 fathers who reported having no job at the time of randomization, for which participation data is available. The outcome of interest is the time between randomization and employment. For the individuals that were only invited to the first interview, the outcome is measured completely if an individual is employed by the time of the survey and censored at the time of the interview otherwise. For the fathers that were invited to the second follow-up interview and participated, the outcome is measured completely if an individual is employed by the time of the second follow-up interview, but is otherwise censored at the second interview date. If the individual does not participate in the second survey after being invited, they will be censored at the time of the first interview. It follows that there could be some dependence between $T$ and $C$ when this decision to go to the second follow-up interview is influenced by them having found a job between the two interview dates. This possible dependence combined with the fact that the data suffer from two-sided noncompliance makes it an appropriate application of the proposed methodology.

The instrument $\tilde{W}$ will be a binary variable indicating whether an individual is in the control or treatment group (0 and 1 respectively). The confounded variable $Z$ indicates whether they actually participated in a JTPA program (0 for no participation and 1 otherwise). This participation variable is confounded due to individuals moving themselves between the treatment and control group in a non-random way. The covariates include the participant's age, race (white or non-white), marital status and whether they have a high school diploma or GED. We expect $\tilde{W}$ to be a valid instrument because it is randomly assigned, correlated with JTPA participation and should have no impact on time until employment other than through participation in a JTPA funded program. Rows 2 through 5 from Table \ref{sumstat} show that the individual characteristics are balanced across the control and treatment group. This indicates satisfactory random assignment. The first row shows that about 31\% of the total sample was assigned to the control group. 

\newpage
\vfill
\begin{figure}[ht]
	\begin{minipage}{0.47\textwidth}
		\begin{figure}[H]
			\centering
			\includegraphics[width=\linewidth]{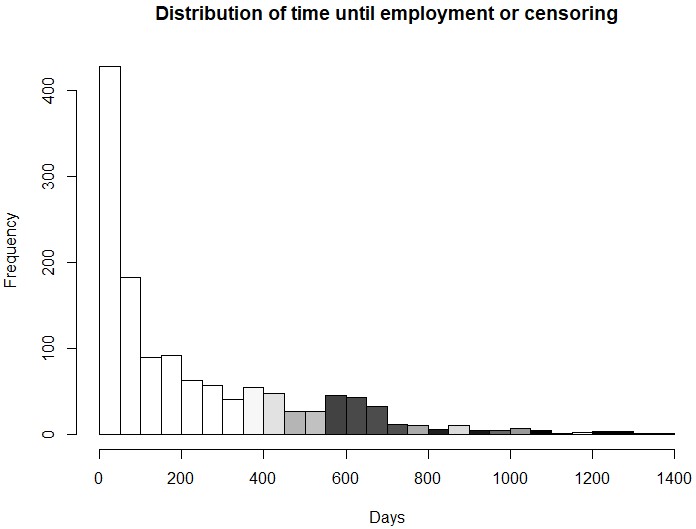}
			\captionof{figure}{}
			\label{distY}
		\end{figure} 
	\end{minipage}
	\hfil
	\begin{minipage}{0.52\textwidth}
		\begin{table}[H]
			\centering
			\scalebox{0.93}{
				\begin{tabular}{|c|c|c|c|}
					\hline
					& total & control & treatment \\
					\hline
					sample size & 1298  &402  & 896 \\
					\hline
					age & 32.7 years  & 33.3 years  & 32.5 years \\
					white & 0.66  & 0.64  & 0.67 \\
					married & 0.77  & 0.76  & 0.77 \\
					GED & 0.47  & 0.48  & 0.47 \\
					\hline
					$Z$ & 0.52  & 0.13  & 0.70 \\
					$(Y,\Delta = 1)$ & 160.9 days  & 181.6 days  & 151.9 days \\
					$\Delta$ & 0.87 & 0.84  & 0.88 \\
					\hline
				\end{tabular}
			}
			\captionof{table}{}
			\label{sumstat}
		\end{table}
	\end{minipage}
	\caption*{Figure 2: The histogram of $Y$, in days, starting from random assignment. The sample includes fathers unemployed at the time of random assignment. The darker the shade of the bin, the higher the censoring rate. 
		
		\smallskip Table 2: Summary statistics for the sample of fathers unemployed at the time of random assignment. The means are shown for the total sample, the control group and the treatment group.}
\end{figure}
\vfill
\begin{table}[H]
	\centering
	\resizebox{\textwidth}{!}{
		\begin{tabular}{|c|rrr|rrr|rrr|}
			\hline
			& \multicolumn{3}{c|}{naive estimator} & \multicolumn{3}{c|}{independent estimator} & \multicolumn{3}{c|}{two-step estimator}  \\
			\hline
			$T$ & Estimate & SE &  $p$-value &  Estimate & SE &  $p$-value &  Estimate & SE &  $p$-value  \\
			\hline
			Intercept & 4.753 & 0.223  & 0.000 &  4.949 & 0.318  & 0.000 &  4.866 & 0.370  & 0.000  \\ 
			Age & 0.015 & 0.006  & 0.007 &  0.013 & 0.006 & 0.025 &  0.015 & 0.008  & 0.056 \\ 
			White & -0.197 & 0.108 & 0.068  & -0.197 & 0.118 & 0.096 & -0.187 & 0.622  & 0.764 \\ 
			Married & -0.331 & 0.123  & 0.007 &  -0.330 & 0.133 & 0.013 &  -0.331 & 0.132 & 0.012 \\ 
			GED & -0.166 & 0.102 & 0.104 &  -0.179 & 0.112 & 0.109 & -0.172 & 0.217 & 0.430 \\
			$\alpha_T$ & -0.218 & 0.102  & 0.033 & -0.483 & 0.260 & 0.063 & -0.428 & 0.210 & 0.041 \\
			$\lambda_T$  &  &  &  & -0.113 & 0.116  & 0.330 & -0.097 & 0.087 & 0.268  \\ 
			\hline
			$C$ & Estimate & SE & $p$-value &  Estimate & SE & $p$-value & Estimate & SE & $p$-value  \\ 
			\hline
			Intercept & 6.866 & 0.134 & 0.000 & 6.672 & 0.164 & 0.000 & 6.848 & 0.403 & 0.000 \\ 
			Age & -0.001 & 0.002 & 0.436 &  -0.001 & 0.002 & 0.605 & -0.001 & 0.003 & 0.605 \\ 
			White & 0.012 & 0.033 & 0.713 & -0.005 & 0.050  & 0.913 & 0.010 & 1.206  & 0.993  \\
			Married & 0.010 & 0.032  & 0.746 & -0.006 & 0.075 & 0.933 & 0.013 & 0.661 & 0.984  \\
			GED & -0.063 & 0.042 & 0.137 &  -0.069 & 0.041 & 0.093 & -0.062 & 0.177 & 0.728  \\  
			$\alpha_C$ & -0.062 & 0.042 & 0.138 & -0.052 & 0.117 & 0.655 & -0.028 & 1.335 & 0.983 \\ 
			$\lambda_C$ &  &  &  & 0.006 & 0.042 & 0.876 & 0.015 & 0.511 & 0.976 \\ 
			\hline
			$\sigma_T$ & 1.817 & 0.040 & 0.000 & 1.804 & 0.038 & 0.000 & 1.816 & 0.038 & 0.000  \\ 
			$\sigma_C$  & 0.323 & 0.037 & 0.000 &  0.285 & 0.013 & 0.000 & 0.323 & 0.029 & 0.000 \\ 
			$\rho$ & -0.430 & 0.196 & 0.028 &   &  &  & -0.432 & 0.176 & 0.014 \\ 
			\hline
		\end{tabular}
	}
	\caption{Estimation results for the naive, independent and two-step estimator. Given are the parameter estimate, standard error (SE) and the $p$-value.}
	\label{tableresults}
\end{table}
\vfill
\newpage

The last 3 rows of table \ref{sumstat} show summary statistics for variables observed after randomization. It is interesting to note that 13\% of the fathers in the control group were nevertheless able to participate in JTPA services compared to 3\% for the entire control group. The mean time to employment also seems to be about 30 days shorter for the individuals assigned to the treatment group compared to the control group. The censoring rate is similar for both groups. Figure \ref{distY} plots a histogram of the observed follow-up time $Y$, where darker shading indicates a higher censoring rate. A lot of the censored observations are around the 600 days mark, at which time most of the first follow-up interviews took place. Since everyone in the sample participated in the first follow-up interview, an observation before the date of the first follow-up survey cannot be censored. 

The results of applying the two-step estimator (using a logit model for $Z$), compared to other estimators, can be found in Table \ref{tableresults}. The naive estimator, which does not treat $Z$ as a confounded variable, seems to underestimate the effect of JTPA services on time until employment compared to the proposed two-step estimator. At a 5\% significance level, both of these estimators find a significant effect of JTPA training reducing time until employment. However, the two-step estimate is almost twice the naive estimate which indicates that the individuals participating in the treatment are those with a lower ability to find employment. The independent estimator, which assumes independent censoring, seems to slightly overestimate the size of the effect compared to the proposed two-step estimator, but is not significant at a 5\% significance level. Age seems to be (borderline) significant across the estimators as does marriage status and having a high school diploma or GED. Being older seems to increase time until employment, while being married and having a diploma reduces it. Both the naive and the two-step estimator seem to agree that there is a quite strong negative correlation of about -0.43 between $T$ and $C$.

\section*{Future research}
It is important to note that this work is only a first step towards a set of models that will allow for the estimation of causal effects under dependent censoring. A first extension could be to select different parametric marginals and copulas by making use of an information criterion. Up until now, the association parameter has been shown to be identified only for certain combinations of parametric copulas and marginals without confounding or covariates (see \cite{CzadoClaudia2021Dcbo}). Implementing this would therefore complicate the identifiability proof and the estimation procedure even further. Another line of research is to allow for semi-parametric marginals. This would greatly increase the flexibility of the model and is currently being studied.

\bibliography{DCCarXiv}

\newpage

\appendix

\appendixpage

\section{Control function examples}\label{CF}
In this first section of the Appendix, we expand on the examples of $g_\gamma(Z,W)$ given in Section \ref{sectioncontfunc}. More specifically, given the distribution of $\nu$, we give explicit expressions for
$$
g_\gamma(Z,W)= Z\EX[\nu   \mid  W^\top{\gamma} > \nu ] + (1-Z)\EX[\nu   \mid   W^\top{\gamma} < \nu ].
$$
If we assume that $\nu $ follows a standard normal distribution, it can be derived that
$$
g_\gamma(Z,W)=  (1-Z)\frac{\phi(W^\top{\gamma})}{\Phi(-W^\top{\gamma})}-Z\frac{\phi(W^\top{\gamma})}{\Phi(W^\top{\gamma})}.
$$
Similarly, if we assume that $\nu $ follows a standard logistic distribution, we find:
\begin{align*}
	g_\gamma(Z,W)&=(1-Z)\Big[\big(1+\exp\{W^\top{\gamma}\}\big)\log\big(1+\exp\{W^\top{\gamma}\}\big)-W^\top{\gamma}\exp\{W^\top{\gamma}\}\Big] \\&-Z\Big[\big(1+\exp\{-W^\top{\gamma}\}\big)\log\big(1+\exp\{-W^\top{\gamma}\}\big)+W^\top{\gamma}\exp\{-W^\top{\gamma}\}\Big].
\end{align*}
Both of these expressions follow from the following conditional expectations. 
Let $\nu \sim N(0,1)$, it follows that:
\begin{align*}
	\EX[\nu  \mid  \nu <a] = \frac{1}{\Phi(a)}\int^a_{-\infty} \nu \phi(\nu )\diff \nu= \frac{-1}{\Phi(a)}\int^a_{-\infty} \phi'(\nu)\diff \nu  = \frac{-\phi(a)}{\Phi(a)},
\end{align*}
\begin{align*}
	\EX[\nu \mid  \nu>a] = \frac{1}{1-\Phi(a)}\int^{+\infty}_a \nu\phi(\nu)\diff \nu
	= \frac{-1}{\Phi(-a)}\int^{+\infty}_a\phi'(\nu)\diff \nu
	= \frac{\phi(a)}{\Phi(-a)}.
\end{align*}
If $\nu\sim$ Logistic$(0,1)$, it follows that:
\begin{align*}
	\EX[\nu \mid  \nu<a]=(1+e^{-a})\int^a_{-\infty} \nu\frac{e^{-\nu}}{(1+e^{-\nu})^2}\diff \nu 
	= -(1+e^{-a})\log(1+e^{-a})-ae^{-a},
\end{align*}
\begin{align*}
	\EX[\nu \mid  \nu>a] = (1+e^{a})\int^{+\infty}_a \nu\frac{e^{-\nu}}{(1+e^{-\nu})^2}\diff \nu  = (1+e^a)\log(1+e^{a})-ae^a.
\end{align*}

\section{Technical lemmas}\label{lemmas}
In this section, we prove three lemmas of which Lemma \ref{lemma_uc} is needed to prove Theorem \ref{thrmcons} in Appendix \ref{Theorems}. Lemma \ref{lemma1} is used in the proof of Lemma \ref{lemma_jad}, and Lemma \ref{lemma_jad} is used in the proof of Lemma \ref{lemma_uc}. Note that the proofs of the first two lemmas are inspired by the proof of Lemma 1 by \citet{TauchenGeorge1985Dtae}. For all of these proofs, it is useful to define the open cube $I(\gamma,\theta,d)$ as
$$
I(\gamma,\theta,d) = \Bigg\{(\Tilde{\gamma},\Tilde{\theta})\in \mathcal{N}_\gamma\times \Theta: \lvert\lvert(\Tilde{\gamma},\Tilde{\theta}) - (\gamma,\theta) \rvert\rvert_\infty < d\Bigg\},
$$
where $\lvert\lvert\cdot\rvert\rvert_\infty$ is the sup-norm and $d \in \mathbb{R}_{>0}$.

\begin{lemma}\label{lemma1}
	If Assumptions \ref{A1}-\ref{A8} hold, then for all $\varepsilon>0$ and $(\gamma,\theta) \in \mathcal{N}_\gamma \times \Theta$, there exists a $d>0$ such that
	$$
	\EX\Bigg[\sup_{(\Tilde{\gamma},\Tilde{\theta}) \in I(\gamma,\theta,d)} \lvert\ell(S,\Tilde{\gamma},\Tilde{\theta})-\ell(S,\gamma,\theta)\rvert\Bigg] < \frac{\varepsilon}{4}.
	$$
\end{lemma}

\begin{proof}
	\noindent Define $$\kappa (s,\gamma,\theta,d)=\sup_{(\Tilde{\gamma},\Tilde{\theta}) \in I(\gamma,\theta,d)} \lvert\ell(s,\Tilde{\gamma},\Tilde{\theta})-\ell(s,\gamma,\theta)\rvert,$$
	with $(s,\gamma,\theta,d) \in \mathcal{G} \times \mathcal{N}_\gamma \times \Theta \times \mathbb{R}_{>0}$. By Assumption \ref{A8}, for all $ (s,\gamma,\theta,d) \in \mathcal{G} \times \mathcal{N}_\gamma \times \Theta \times \mathbb{R}_{>0}:$
	\begin{align*}
		0 \leq \kappa (s,\gamma,\theta,d) & \leq \sup_{\Tilde{\gamma} \in \mathcal{N}_\gamma, \Tilde{\theta} \in \Theta} \lvert\ell(s,\Tilde{\gamma},\Tilde{\theta})-\ell(s,\gamma,\theta)\rvert \leq 2\cdot\mathcal{D}(s).
	\end{align*}
	Let $d_n$ be a sequence that converges to $0$ when $n \to \infty$ and define $\kappa_n (s,\gamma,\theta) = \kappa (s,\gamma,\theta,d_n).$ The continuity of $\ell$, implies that (use Heine-Cantor theorem) for all $(s,\gamma,\theta) \in \mathcal{G} \times \mathcal{N}_\gamma \times \Theta$ it holds that 
	$$
	\lim_{n \to \infty} \kappa_n (s, \gamma,\theta) = 0.
	$$
	Since $\EX\big[\mathcal{D}(s)\big] < \infty$, the dominated convergence theorem states that 
	$$\lim_{n \to \infty}\EX\Bigg[\sup_{(\Tilde{\gamma},\Tilde{\theta}) \in I(\gamma,\theta, d_n)} \lvert\ell(S,\Tilde{\gamma},\Tilde{\theta})-\ell(S,\gamma,\theta)\rvert\Bigg] = 0.$$
	From this we can infer that for all $\varepsilon>0$ and $(\gamma,\theta) \in \mathcal{N}_\gamma \times \Theta$, there exists a $d>0$ such that
	$$
	\EX\Bigg[\sup_{(\Tilde{\gamma},\Tilde{\theta}) \in I(\gamma,\theta,d)} \lvert \ell(S,\Tilde{\gamma},\Tilde{\theta})-\ell(S,\gamma,\theta)\rvert \Bigg] < \frac{\varepsilon}{4}.
	$$
\end{proof}

\begin{lemma}\label{lemma_jad}
	Under Assumptions \ref{A1}-\ref{A8}, there exists a neighborhood $B$ of $\gamma^*$ such that for all $ \varepsilon > 0$ and $\gamma \in B$, we have
	
	\begin{equation}\label{eq_Jad}
		\lim_{n \to \infty} \mathbb{P}\Bigg(\sup_{\gamma\in B,\theta \in \Theta} \big\lvert\hat{L}(\gamma,\theta)-L(\gamma^*,\theta)\big\rvert < \varepsilon \Bigg) = 1.
	\end{equation}
\end{lemma}
\begin{proof}
	
	Take $\gamma \in \mathcal{N}_\gamma$ (with $\mathcal{N}_\gamma$ defined in Assumption \ref{A8}). By Lemma 1, for all $\theta \in \Theta$, there exists a $d_\theta>0$ such that 
	\begin{equation}\label{lemma1res}
		\EX\Bigg[\sup_{( \Tilde{\gamma},\Tilde{\theta}) \in I(\gamma^*,\theta,d_\theta)} \lvert \ell(S, \Tilde{\gamma},\Tilde{\theta})-\ell(S, \gamma^*,\theta)\rvert \Bigg] < \frac{\varepsilon}{4}.
	\end{equation}
	Given these open cubes, we can define 
	$$
	I = \bigcup_{\theta \in \Theta} I(\gamma^*,\theta,d_\theta),
	$$
	which is clearly an open cover of $\{\gamma^*\} \times \Theta$. 
	
	\newpage \noindent Because $\Theta$ is assumed to be compact and $\gamma^*$ fixed, the open cover definition of a compact space tells us that since $I$ is a collection of open subsets that cover $\{\gamma^*\} \times \Theta$, there must exist a finite subcollection $ \{I_{1},I_{2},...,I_{K}\}$ such that 
	$$
	\{\gamma^*\} \times \Theta \subseteq \bigcup_{k=1\dots, K}I_k,
	$$
	with 
	$$
	I_k=I(\gamma^*,\theta_k,d_k),\ k=1,\dots, K,
	$$ 
	for some $d_k>0$.
	Thanks to the subcollection being finite, we can define $\bar{d} = \min_{k=1,...,K} d_k>0.$ Let $B_\gamma$ be the open ball in $\Gamma$ centered around $\gamma^*$ with a radius of $\Bar{d}$. For all $ k=1,...,K$ and $\gamma \in B_\gamma$, it holds that $(\gamma,\theta_k) \in I_k$, which means that $F$ is also a finite open cover of $B_\gamma \times \Theta$. 
	Let $B =B_\gamma \cap \mathcal{N}_\gamma$, $\gamma \in B$ and $(\gamma, \theta) \in I_{k}$. It follows from \eqref{lemma1res} that
	
	\begin{align*}
		\frac{\varepsilon}{4} & \geq \EX\Big[ \lvert\ell(S, \gamma,\theta)-\ell(S, \gamma^*,\theta_k)\rvert\Big]  \geq \lvert L( \gamma,\theta)-L(\gamma^*,\theta_k)\rvert.
	\end{align*}
	This means that
	\begin{equation}\label{bound1}
		\lvert L(\gamma,\theta)-L(\gamma^*,\theta_k)\rvert \leq \frac{\varepsilon}{4} \quad \text{ when } (\gamma,\theta) \in I_{k}.
	\end{equation}
	We have
	\begin{flalign}
		&\big\lvert\hat{L}(\gamma,\theta)-L(\gamma^*,\theta)\big\rvert \nonumber \\
		&= \bigg\lvert\frac{1}{n}{\displaystyle \sum_{i=1}^{n}\ell(S_i, \gamma,\theta)-\ell(S_i, \gamma^*,\theta_k)+\ell(S_i, \gamma^*,\theta_k)}-L(\gamma^*,\theta_k)+L(\gamma^*,\theta_k)-L(\gamma^*,\theta)\bigg\rvert \nonumber 
		\\& \leq \frac{1}{n}{\displaystyle \sum_{i=1}^{n}\Big\lvert\ell(S_i, \gamma,\theta)-\ell(S_i, \gamma^*,\theta_k)}\Big\rvert   -  \EX\Bigg[\sup_{(\Tilde{\gamma},\Tilde{\theta}) \in I_{k}} \lvert\ell(S,\Tilde{\gamma},\Tilde{\theta})-\ell(S, \gamma^*,\theta_k)\rvert \Bigg] \nonumber \\ & + \EX\Bigg[\sup_{(\Tilde{\gamma},\Tilde{\theta}) \in I_{k}} \lvert \ell(S,\Tilde{\gamma},\Tilde{\theta})-\ell(S, \gamma^*,\theta_k)\rvert \Bigg] +\big\lvert\hat{L}(\gamma^*,\theta_k)-L( \gamma^*,\theta_k)\big\rvert \nonumber +\Big\lvert L(\gamma^*,\theta_k)-L(\gamma^*,\theta)\Big\rvert \nonumber \\& \leq  \frac{1}{n}{\displaystyle \sum_{i=1}^{n}\sup_{(\Tilde{\gamma},\Tilde{\theta}) \in I_{k}}\lvert \ell(S_i,\Tilde{\gamma},\Tilde{\theta})-\ell(S_i, \gamma^*,\theta_k)}\rvert - \EX\Bigg[\sup_{(\Tilde{\gamma},\Tilde{\theta}) \in I_{k}} \lvert \ell(S,\Tilde{\gamma},\Tilde{\theta})-\ell(S, \gamma^*,\theta_k)\rvert\Bigg] \label{part1}  \\& + \EX\Bigg[\sup_{(\Tilde{\gamma},\Tilde{\theta}) \in I_{k}} \lvert \ell(S,\Tilde{\gamma},\Tilde{\theta})-\ell(S, \gamma^*,\theta_k)\rvert \Bigg] +\Big\lvert L( \gamma^*,\theta_k)-L(\gamma^*,\theta)\Big\rvert. \label{part2} \\& +\big\lvert\hat{L}(\gamma^*,\theta_k)-L( \gamma^*,\theta_k)\big\rvert. \label{part3}
	\end{flalign}
	The terms in \eqref{part2} are each bounded by $\frac{\varepsilon}{4}$ due to \eqref{lemma1res} and \eqref{bound1} respectively. For \eqref{part1} and \eqref{part3}, the strong law of large numbers can be applied such that they go to $0$ in probability. Because of this, we have that:
	\begin{equation*}
		\lim_{n \to \infty} \mathbb{P}\Bigg(\sup_{(\gamma,\theta) \in I_k}\bigg\lvert \hat{L}(\gamma,\theta)-L(\gamma^*,\theta)\bigg\rvert < \varepsilon \Bigg)=1,
	\end{equation*}
	for all $k=1,\dots,K$. 
	Since $F$ is a finite open covering of $B\times \Theta$, this leads to \eqref{eq_Jad}.\\
\end{proof}

\begin{lemma}If Assumptions \ref{A1}-\ref{A8} hold, then
	$$
	\sup_{\theta \in \Theta} \big\lvert \hat{L}(\hat{\gamma},\theta)-L(\gamma^*,\theta)\big\rvert \xrightarrow{\text{ p }}0.
	$$\label{lemma_uc}
\end{lemma}

\begin{proof}
	Note that for all $\varepsilon > 0$:
	\begin{align}
		\notag & \mathbb{P}\Bigg(\sup_{\theta \in \Theta} \bigg\lvert \hat{L}(\hat{\gamma},\theta)-L( \gamma^*,\theta)\bigg\rvert > \varepsilon \Bigg) \\ \notag & = \mathbb{P}\Bigg(\left\{\sup_{\theta \in \Theta} \bigg\lvert\hat{L}(\hat{\gamma},\theta)-L(\gamma^*,\theta)\bigg\rvert > \varepsilon \right\}\quad \bigcap \quad \Big\{\hat{\gamma} \in B \Big\} \Bigg) \\ &  + \mathbb{P}\Bigg(\left\{\sup_{\theta \in \Theta} \bigg\lvert \hat{L}(\hat{\gamma},\theta)-L(\gamma^*,\theta)\bigg\rvert > \varepsilon \right\}\quad \bigcap \quad \Big\{\hat{\gamma} \notin B \Big\} \Bigg).\label{pigeonhole}
	\end{align}
	Assumption \ref{A6} implies that $\hat{\gamma}\xrightarrow{\text{ p }}\gamma^{*}$. Because of this we know that
	$$
	\lim_{n \to \infty} \mathbb{P}\Bigg(\left\{\sup_{\theta \in \Theta} \bigg\lvert\hat{L}(\hat{\gamma},\theta)-L(\gamma^*,\theta)\bigg\rvert > \varepsilon\right\} \quad \bigcap \quad \Big\{\hat{\gamma} \notin B \Big\} \Bigg) \leq \lim_{n \to \infty} \mathbb{P}\big(\hat{\gamma} \notin B \big) = 0.
	$$
	Moreover, we have
	\begin{align*}
		&\mathbb{P}\Bigg(\left\{\sup_{\theta \in \Theta} \bigg\lvert \hat{L}(\hat{\gamma},\theta)-L(\gamma^*,\theta)\bigg\rvert > \varepsilon \right\}\quad \bigcap \quad \Big\{\hat{\gamma} \in B \Big\} \Bigg)\\
		&\le \mathbb{P}\Bigg(\sup_{(\gamma,\theta) \in B\times \Theta} \bigg\lvert \hat{L}(\gamma,\theta)-L(\gamma^*,\theta)\bigg\rvert > \varepsilon \Bigg)\to 0,
	\end{align*}
	by Lemma \ref{lemma_jad}. 
	We obtain the result using \eqref{pigeonhole}. 
\end{proof}
\section{Proofs of theorems}\label{Theorems}
In this section of the Appendix, the three main theorems regarding identifiability, consistency and asymptotic normality are proven.

\subsection*{Proof of Theorem 1}

Because of Assumption \ref{A5}, we know that $\gamma$ is identified. For ease of notation, let $\theta_1=\theta$ and $\theta_2=\theta^*$. From the model specification, we know that:
\begin{align*}
	& f_{Y_j,\Delta_j \mid W,Z}(y,1 \mid w,z,\gamma ;\theta_j) \\& =  \Bigg[1-\Phi \Bigg(\frac{\Big(1-\rho_j \frac{\sigma_{C_j}}{\sigma_{T_j}}\Big)y-x^\top{\beta_{C_j}} - z\alpha_{C_j} - v\lambda_{C_j} + \rho_j \frac{\sigma_{C_j}}{\sigma_{T_j}}\big(x^\top{\beta_{T_j}} + z\alpha_{T_j} + v\lambda_{T_j}\big)}{\sigma_{C_j}(1-{\rho_j}^2)^{\sfrac{1}{2}}}\Bigg)\Bigg] \\ &\times \frac{1}{\sigma_{T_j}} \phi\bigg(\frac{y-x^\top{\beta_{T_j}} - z\alpha_{T_j} - v\lambda_{T_j}}{\sigma_{T_j}}\bigg), \hspace{0.4cm} j=1,2.
\end{align*}
We will consider a number of cases that are dependent on the values of the following $\pi$'s: $$\pi_{11} = 1 - \rho_1 \frac{\sigma_{C_1}}{\sigma_{T_1}}, \hspace{0.2cm} \pi_{12} = 1 - \rho_1 \frac{\sigma_{T_1}}{\sigma_{C_1}}, \hspace{0.2cm} \pi_{21} = 1 - \rho_2 \frac{\sigma_{C_2}}{\sigma_{T_2}}, \hspace{0.2cm} \pi_{22} = 1 - \rho_2 \frac{\sigma_{T_2}}{\sigma_{C_2}}.$$
\newpage
\noindent \emph{Case 1:} All $\pi_{jk} \text{ with }j,k = 1,2$ are strictly positive. Note that the positivity of $\pi_{j1}$ $(j=1,2)$ allows us to rewrite the argument of $\Phi(\cdot)$ as:

$$
\frac{y  - \frac{x^\top{\beta_{C_j}} + z\alpha_{C_j} + v\lambda_{C_j} - \rho_j \frac{\sigma_{C_j}}{\sigma_{T_j}}\big(x^\top{\beta_{T_j}} + z\alpha_{T_j} + v\lambda_{T_j}\big)}{\pi_{j1}}}{\sqrt{\frac{\sigma_{C_j}^2(1-\rho_j^2)}{\pi_{j1}^2}}}.
$$
Therefore, we define $\xi_{j1}$ $(j = 1,2),$ whose distribution for a given $(W,Z)$ is specified as:
$$
(\xi_{j1} \mid W, Z)\sim N \Bigg(\frac{x^\top{\beta_{C_j}} + z\alpha_{C_j} + v\lambda_{C_j} - \rho_j \frac{\sigma_{C_j}}{\sigma_{T_j}}\big(x^\top{\beta_{T_j}} + z\alpha_{T_j} + v\lambda_{T_j}\big)}{\pi_{j1}},\frac{\sigma_{C_j}^2(1-\rho_j^2)}{\pi_{j1}^2}\Bigg).
$$
This allows us to rewrite the sub-density as follows:
\begin{equation}\label{prob1}
	f_{Y_j,\Delta_j \mid W,Z}(y,1 \mid w,z,\gamma ;\theta_j)=\mathbb{P}(\xi_{j1} > y \mid W=w, Z=z)f_{T_j \mid W,Z}(y \mid w,z,\gamma;\theta_j).
\end{equation}
Since $f_{Y_1,\Delta_1 \mid X,Z,V}(y,1 \mid w,z,\gamma ;\theta_1) = f_{Y_2,\Delta_2 \mid W,Z}(y,1 \mid w,z,\gamma ;\theta_2)$ for almost every $(w,z)$, it follows from \eqref{prob1} that $$\lim_{y\to-\infty}\frac{f_{T_1 \mid W,Z}(y \mid w,z,\gamma;\theta_1)}{ f_{T_2 \mid W,Z}(y \mid w,z,\gamma;\theta_2)} = 1 \text{ for almost every }(w,z).$$ Because of Proposition A.1. by \citet{DeresaNegeraWakgari2020Fpmf}, it is implied that: $$\beta_{T_1} = \beta_{T_2}, \hspace{0.2cm}\alpha_{T_1} = \alpha_{T_2}, \hspace{0.2cm}\lambda_{T_1} = \lambda_{T_2}, \hspace{0.2cm} \sigma_{T_1} = \sigma_{T_2}.$$ Putting $\Delta_j=0$ and repeating the same arguments (with $\pi_{j2}>0,\hspace{0.1cm}j=1,2)$, we get that: $$\beta_{C_1} = \beta_{C_2}, \hspace{0.2cm}\alpha_{C_1} = \alpha_{C_2}, \hspace{0.2cm}\lambda_{C_1} = \lambda_{C_2}, \hspace{0.2cm} \sigma_{C_1} = \sigma_{C_2}.$$ 
Using the expression for $F_{Y \mid W,Z}(y \mid w,z,\gamma ;\theta)$ from Section \ref{secdistr}, we know that:
\begin{align*}
	&\Phi\bigg(\frac{y-x^\top{\beta_{T_1}} - z\alpha_{T_1} - v\lambda_{T_1}}{\sigma_{T_1}}, \frac{y-x^\top{\beta_{C_1}} - z\alpha_{C_1} - v\lambda_{C_1}}{\sigma_{C_1}}; \rho_1\bigg) \\&= \Phi\bigg(\frac{y-x^\top{\beta_{T_2}} - z\alpha_{T_2} - v\lambda_{T_2}}{\sigma_{T_2}}, \frac{y-x^\top{\beta_{C_2}} - z\alpha_{C_2} - v\lambda_{C_2}}{\sigma_{C_2}}; \rho_2\bigg),
\end{align*}
for almost every $(w,z)$. From this it is clear that $\rho_1 = \rho_2,$ and thus $\theta_1=\theta_2$.

\bigskip \noindent \emph{Case 2:} One of $\pi_{jk} \text{ with }j,k = 1,2$ is strictly negative and the others are strictly positive. Firstly, we assume that  $\pi_{11} < 0,$ and therefore it must be the case that $\pi_{21} > 0.$ From \emph{Case 1} we know that the positivity of $\pi_{21}$ implies that
$$
f_{Y_2,\Delta_2 \mid W,Z}(y,1 \mid w,z,\gamma ;\theta_2)=\mathbb{P}(\xi_{21} > y \mid W=w, Z=z)f_{T_2 \mid W,Z}(y \mid w,z,\gamma;\theta_2).
$$
Note that the negativity of $\pi_{11}$ allows us to rewrite the argument of $\Phi(\cdot)$ as
$$
\frac{-y  - \frac{-\Big[x^\top{\beta_{C_1}} + z\alpha_{C_1} + v\lambda_{C_1} - \rho_1 \frac{\sigma_{C_1}}{\sigma_{T_1}}\big(x^\top{\beta_{T_1}} + z\alpha_{T_1} + v\lambda_{T_1}\big)\Big]}{\pi_{11}}}{\sqrt{\frac{\sigma_{C_1}^2(1-\rho_1^2)}{\pi_{11}^2}}}.
$$
As before, we define a variable $\zeta_{j1}$ $(j = 1,2)$ whose distribution for given $(W,Z)$ is specified as:
$$
(\zeta_{j1} \mid W, Z)\sim N \Bigg(\frac{-\Big[x^\top{\beta_{C_j}} + z\alpha_{C_j} + v\lambda_{C_j} - \rho_j \frac{\sigma_{C_j}}{\sigma_{T_j}}\big(x^\top{\beta_{T_j}} + z\alpha_{T_j} + v\lambda_{T_j}\big)\Big]}{\pi_{j1}},\frac{\sigma_{C_j}^2(1-\rho_j^2)}{\pi_{j1}^2}\Bigg).
$$
This can be used to find that:
\begin{align*}
	f_{Y_1,\Delta_1 \mid W,Z,}(y,1 \mid w,z,\gamma ;\theta_1)&=\mathbb{P}(\zeta_{11} > -y \mid W=w, Z=z)f_{T_1 \mid W,Z}(y \mid w,z,\gamma;\theta_1), \\& =\mathbb{P}(-\zeta_{11} < y \mid W=w, Z=z)f_{T_1 \mid W,Z}(y \mid w,z,\gamma;\theta_1),
	\\&=\mathbb{P}(\xi_{11} < y \mid W=w, Z=z)f_{T_1 \mid W,Z}(y \mid w,z,\gamma;\theta_1).
\end{align*}
Because $f_{Y_1,\Delta_1 \mid W,Z}(y,1 \mid w,z,\gamma ;\theta_1) = f_{Y_2,\Delta_2 \mid W,Z}(y,1 \mid w,z,\gamma ;\theta_2),$ we have that:
$$
\mathbb{P}(\xi_{21} > y \mid W=w, Z=z) = \mathbb{P}(\xi_{11} < y \mid W=w, Z=z)\times \frac{f_{T_1 \mid W,Z}(y \mid w,z,\gamma;\theta_1)}{f_{T_2 \mid W,Z}(y \mid w,z,\gamma;\theta_2)}.
$$
Taking the limit on both sides when $y$ approaches $-\infty$, it follows that the left-hand side goes to 1. However, Lemma 2.4 by \citet{BasuA.P1978Iotm} shows that the right hand side does not go to 1, leading to a contradiction. Using the same arguments, it can be shown that one strictly negative $\pi$ and three strictly positive $\pi$'s always leads to a contradiction.

\bigskip \noindent \emph{Case 3:} Two of $\pi_{jk} \text{ with }j,k = 1,2$ are strictly negative and the other two are strictly positive. When $\pi_{11} \text{ and } \pi_{12}$ are both either strictly positive or strictly negative, \emph{Case 2} shows that this leads to a contradiction. Therefore, one of $\pi_{11}, \pi_{12}$ and one of $\pi_{21}, \pi_{22}$ is strictly positive. Assuming $\pi_{11}>0$ and $\pi_{12} < 0,$ we know from \emph{Case 2} that if $\pi_{21} < 0$ we get a contradiction. Hence, we further assume that $\pi_{21}>0 \text{ and } \pi_{22} < 0.$ We now define 
$$
(\zeta_{j2} \mid W, Z)\sim N \Bigg(\frac{-\Big[x^\top{\beta_{T_j}} + z\alpha_{T_j} + v\lambda_{T_j} - \rho_j \frac{\sigma_{T_j}}{\sigma_{C_j}}\big(x^\top{\beta_{C_j}} + z\alpha_{C_j} + v\lambda_{C_j}\big)\Big]}{\pi_{j2}},\frac{\sigma_{T_j}^2(1-\rho_j^2)}{\pi_{j2}^2}\Bigg),
$$
$$
(\xi_{j2} \mid W, Z)\sim N \Bigg(\frac{\Big[x^\top{\beta_{T_j}} + z\alpha_{T_j} + v\lambda_{T_j} - \rho_j \frac{\sigma_{T_j}}{\sigma_{C_j}}\big(x^\top{\beta_{C_j}} + z\alpha_{C_j} + v\lambda_{C_j}\big)\Big]}{\pi_{j2}},\frac{\sigma_{T_j}^2(1-\rho_j^2)}{\pi_{j2}^2}\Bigg).
$$
Since $f_{Y_1,\Delta_1 \mid W,Z}(y,0 \mid w,z,\gamma ;\theta_1) = f_{Y_2,\Delta_2 \mid W,Z}(y,0 \mid w,z,\gamma ;\theta_2),$ the same arguments as before can be used to show that:
\begin{align*}
	\mathbb{P}(\xi_{12} < y \mid W=w, Z=z)\times\frac{f_{C_1 \mid W,Z}(y \mid w,z,\gamma;\theta_1)}{f_{C_2 \mid W,Z}(y \mid w,z,\gamma;\theta_2)}=\mathbb{P}(\xi_{22} < y \mid W=w, Z=z).
\end{align*}
It follows that $$\lim_{y\to+\infty}\frac{f_{C_1 \mid W,Z}(y \mid w,z,\gamma;\theta_1)}{f_{C_2 \mid W,Z}(y \mid w,z,\gamma;\theta_2)} = 1 \text{ for almost every }(w,z).$$ Application of Proposition A.1. by \citet{DeresaNegeraWakgari2020Fpmf} implies that $$\beta_{C_1} = \beta_{C_2}, \hspace{0.2cm}\alpha_{C_1} = \alpha_{C_2}, \hspace{0.2cm}\lambda_{C_1} = \lambda_{C_2}, \hspace{0.2cm} \sigma_{C_1} = \sigma_{C_2}.$$ The result for $\Delta_j = 1$ and $\pi_{11}, \pi_{21} > 0$ was already discussed in \emph{Case 1}. Combining these, it is clear that $\theta_1=\theta_2$.
It follows that this argument can be replicated for $ \pi_{11}, \pi_{21} < 0$ and $\pi_{12}, \pi_{22} > 0.$

\bigskip \noindent \emph{Case 4:} Three or four of $\pi_{jk} \text{ with }j,k = 1,2$ are strictly negative and one or none of the $\pi$'s are strictly positive, respectively. This immediately leads to a contradiction since
$$
\rho_1 \frac{\sigma_{C_1}}{\sigma_{T_1}}\rho_1 \frac{\sigma_{T_1}}{\sigma_{C_1}}=\rho_1^2 < 1,
$$
which implies that $$\rho_1 \frac{\sigma_{C_1}}{\sigma_{T_1}} <1 \text{ or } \rho_1 \frac{\sigma_{T_1}}{\sigma_{C_1}} < 1,$$ meaning that at least one of $\pi_{11}$ and $\pi_{12}$ is strictly positive. This also is the case for $\pi_{21}$ and $\pi_{22}$. Hence, there are always at least two strictly positive $\pi$'s and we have a contradiction.

\bigskip \noindent \emph{Case 5:} One of $\pi_{jk} \text{ with }j,k = 1,2$ is equal to zero. Without loss of generality, we assume that $\pi_{11}=0$. 
Note that $\pi_{11}=0$ implies that:
\begin{align*}
f_{Y_1,\Delta_1 \mid W,Z}(y,1 \mid w,z,\gamma ;\theta_1) = & \Bigg[1-\Phi \Bigg(\frac{-x^\top{\beta_{C_1}} - z\alpha_{C_1} - v\lambda_{C_1} + x^\top{\beta_{T_1}} + z\alpha_{T_1} + v\lambda_{T_1}\big)}{\sigma_{C_1}(1-{\rho_1}^2)^{\sfrac{1}{2}}}\Bigg)\Bigg] \\ & \times \frac{1}{\sigma_{T_1}} \phi\bigg(\frac{y-x^\top{\beta_{T_1}} - z\alpha_{T_1} - v\lambda_{T_1}}{\sigma_{T_1}}\bigg).
\end{align*}
Because $y$ is no longer included in the argument of $\Phi(\cdot)$, we can rewrite this as
$$f_{Y_1,\Delta_1 \mid W,Z}(y,1 \mid w,z,\gamma  ;\theta_1) =p\times\frac{1}{\sigma_{T_1}} \phi\bigg(\frac{y-x^\top{\beta_{T_1}} - z\alpha_{T_1} - v\lambda_{T_1}}{\sigma_{T_1}}\bigg), \text{ with }0< p<1.$$
Using what we have learned from the previous cases, and assuming that $\pi_{21}>0$, we have that
$$
p\times f_{T_1 \mid W,Z}(y \mid w,z,\gamma ;\theta_1)=\mathbb{P}(\xi_{21} > y \mid W=w, Z=z)f_{T_2 \mid W,Z}(y \mid w,z,\gamma;\theta_2).
$$
Taking the limit where $y \to -\infty$ we get that
$$
p = \lim_{y\to-\infty} \frac{f_{T_2 \mid W,Z}(y \mid w,z,\gamma;\theta_2)}{f_{T_1 \mid W,Z}(y \mid w,z,\gamma ;\theta_1)} \text{ for almost every }(w,z).
$$
Note that this is a contradiction since, according to Lemma 2.3 by \citet{BasuA.P1978Iotm}, the right hand side of the equation can only be equal to $0,1 \text{ or } \infty$. It follows that this argument can be replicated when $\pi_{21} < 0$.

\bigskip \noindent \emph{Case 6:} Two of $\pi_{jk}, \text{ with }j,k = 1,2$ are equal to zero. Note that it cannot be the case that $\pi_{j1} \text{ and } \pi_{j2}$ are both zero since $\lvert \rho_j \rvert < 1$ with $j=1,2$. Therefore, one of $\pi_{11}, \pi_{12}$ and one of $\pi_{21}, \pi_{22}$ needs to be zero. To avoid contradictions, as seen in \emph{Case 5}, the only possibilities are $\pi_{11} = \pi_{21} = 0$ or $\pi_{12} = \pi_{22} = 0.$ Without loss of generality, we will assume that $\pi_{11} = \pi_{21} = 0.$ Using the same arguments as in \emph{Case 5}, it follows that
$$
\frac{p_1}{p_2} = \lim_{y\to \pm \infty} \frac{f_{T_2 \mid W,Z}(y \mid w,z,\gamma;\theta_2)}{f_{T_1 \mid W,Z}(y \mid w,z,\gamma ;\theta_1)}\text{ for almost every }(w,z),
$$
with $0<p_1,p_2<1.$ As discussed before, the right hand side can only be equal to $0,1 \text{ or } \infty.$ Therefore, the only way this does not lead to a contradiction is if the right hand side is equal to 1 and $p_1=p_2$. Using Proposition A.1. by \citet{DeresaNegeraWakgari2020Fpmf} combined with $p_1=p_2$, it can easily be shown that $\theta_1 = \theta_2.$

\bigskip \noindent \emph{Case 7:} Three or four of $\pi_{jk}, \text{ with } j,k = 1,2$ are equal to zero. Note that if $\pi_{11} = 0$, it must be the case that $\pi_{12} > 0$. The same holds for $\pi_{21}$ and $\pi_{12}$, leading to a contradiction. \QEDB

\subsection*{Proof of Theorem 2}
Note that (i) $\ell$ is continuous in $\theta$, since it consists of well known other continuous functions and the definition of $\Theta$. Because of the continuity of $\ell$ and Assumptions \ref{A7}-\ref{A8}, it is clear that $L(\gamma,\theta)$ is also continuous in $\theta$. Furthermore, Lemma \ref{lemma_uc} shows that (ii) $\hat{L}(\hat{\gamma},\theta)$ converges uniformly $(\text{in }\theta \in \Theta)$ in probability to $L(\gamma^*,\theta)$ under Assumptions \ref{A1}-\ref{A8}. Moreover, we also have that (iii) $L(\gamma^*,\theta)$ is uniquely maximized at $L(\gamma^*,\theta^*)$.
Indeed, let $\theta\in \Theta$ such that $\theta \ne \theta^*$ and define
$$
R = \frac{f_{Y,\Delta \mid X,Z,V}(Y,\Delta \mid X,Z,W,\gamma^*;\theta)}{f_{Y,\Delta \mid X,Z,V}(Y,\Delta \mid X,Z,W,\gamma^*;\theta^*)}.
$$
Assumptions \ref{A1}-\ref{A5} guarantee that $\theta^*$ is identifiable according to Theorem \ref{thrmident}. Hence, $R$ is not $1$ almost everywhere. As a result, by the strict version of Jensen's inequality, we have
$$
-\text{log}\big(\EX[R]\big) < \EX\big[-\text{log}(R)\big].
$$
This means that
$$
L(\gamma^*,\theta^*) - L(\gamma^*,\theta) = \EX\big[-\text{log}(R)\big] > -\text{log}\big(\EX[R]\big) = -\text{log}(1) = 0,
$$
which shows (iii). Given (i)-(iii) and Assumption \ref{A7}, Theorem 2.1 from \citet{newey1994large} tells us that
$$\hat{\theta}\xrightarrow{\text{ p }}\theta^*.$$ 
Combined with Assumption \ref{A6}, this leads to the desired result.
\QEDB

\subsection*{Proof of Theorem 3}
Since Assumptions \ref{A1}-\ref{A8} guarantee that Theorem \ref{thrmcons} holds, we know that (i) $\hat{\theta}\xrightarrow{\text{ p }}\theta^*$. By Assumption \ref{A6}, using straightforward calculus and the continuous mapping theorem, it can be shown that (ii) $\Tilde{h}(S,\gamma,\theta) = \big(h_m(W,Z,\gamma)^\top, h_\ell(S,\gamma,\theta)^\top  \big)^\top$ is continuously differentiable in a neighborhood $\mathcal{N}_{\gamma,\theta}$ of $(\gamma^*,\theta^*)$ with probability approaching one. Theorem \ref{thrmcons} also shows that $L(\gamma,\theta)$ is uniquely maximized at $\theta^*$ for a given $\gamma$, which combined with Assumptions \ref{A6} and \ref{A7} implies that (iii) $\EX\big[\Tilde{h}(S,\gamma^*,\theta^*)\big]=0$. Given (i)-(iii), combined with Assumptions \ref{A6}, \ref{A7}, \ref{A9} and \ref{A10}, Theorem 6.1 from \citet{newey1994large} tells us that 
$$
\sqrt{n}(\hat{\theta}-\theta^{*}) \xrightarrow{\text{ d }} N(0,\Sigma_\theta) \quad$$ with $$\quad \Sigma_\theta = H^{-1}_\theta \EX\big[\{h_\ell(S,\gamma^{*}, \theta^{*}) + H_\gamma \Psi\}\{h_\ell(S,\gamma^{*}, \theta^{*}) + H_\gamma \Psi\}^\top\big]\big({H^{-1}_\theta}\big)^\top.
$$ \QEDB

\section{Tables}\label{Tables}
\begin{table}[ht]
	\centering
	\caption{Estimation results for design 1 with 51\% censoring and 2500 simulations. Given are the bias, the empirical standard deviation (ESD), the root mean squared error (RMSE) and the confidence rate (CR).}
	\resizebox{\textwidth}{!}{
		\begin{tabular}{|c|rrrr|rrrr|rrrr|}
			\hline
			\multicolumn{5}{|c|}{$n=250$} & \multicolumn{4}{c|}{$n=500$} & \multicolumn{4}{c|}{$n=1000$}\\
			\hline
			\multicolumn{13}{|c|}{naive estimator} \\
			\hline
			& Bias & ESD & RMSE & CR & Bias & ESD & RMSE & CR & Bias & ESD & RMSE & CR \\ 
			\hline
			$\beta_{T,0}$ & -1.798 & 0.219  & 1.811 & 0.000 & -1.805 & 0.153 & 1.811 & 0.000 & -1.807 & 0.104 & 1.810 & 0.000  \\ 
			$\beta_{T,1}$ & -0.827 & 0.186  & 0.848 & 0.011 & -0.834 & 0.126  & 0.844 & 0.000 & -0.833 & 0.090 & 0.838 & 0.000\\ 
			$\alpha_T$ & 1.384 & 0.118 & 1.389 & 0.000 & 1.386 & 0.082 & 1.388 & 0.000 & 1.383 & 0.057 & 1.384 & 0.000 \\ 
			$\beta_{C,0}$ & -1.092 & 0.352  & 1.147 & 0.140 & -1.092 & 0.250 & 1.120 & 0.010 & -1.101 & 0.171 & 1.114 & 0.000 \\ 
			$\beta_{C,1}$ & -0.496 & 0.143  & 0.516 & 0.070 & -0.499 & 0.098 & 0.509 & 0.000 & -0.499 & 0.071 & 0.504 & 0.000  \\ 
			$\alpha_C$ & 0.835 & 0.099  & 0.841 & 0.000 & 0.834 & 0.070 & 0.837 & 0.000  & 0.838 & 0.048 & 0.840 & 0.000 \\ 
			$\sigma_T$ & 1.352 & 0.133  & 1.359 & 0.000 & 1.364 & 0.098 & 1.367 & 0.000  & 1.367 & 0.068 & 1.368 & 0.000  \\ 
			$\sigma_C$ & 0.513 & 0.103  & 0.523 & 0.000 & 0.522 & 0.072 & 0.527 & 0.000 & 0.526 & 0.051 & 0.529 & 0.000 \\ 
			$\rho$ & 0.103 & 0.076  & 0.128 & 0.818 & 0.107 & 0.052 & 0.119 & 0.618  & 0.110 & 0.035 & 0.115 & 0.317 \\ 
			\hline
			\multicolumn{13}{|c|}{independent estimator} \\
			\hline
			$\beta_{T,0}$ & 0.486 & 0.421  & 0.643 & 0.858 & 0.466 & 0.293 & 0.551 & 0.676 & 0.458 & 0.209  & 0.503 & 0.389 \\ 
			$\beta_{T,1}$ & 0.139 & 0.322  & 0.351 & 0.942 & 0.130 & 0.227 & 0.261 & 0.922 & 0.129 & 0.156  & 0.202 & 0.884  \\ 
			$\alpha_T$ & 0.036 & 0.220  & 0.222 & 0.916 & 0.041 & 0.153 & 0.158 & 0.923 & 0.050 & 0.107  & 0.118 & 0.904 \\ 
			$\lambda_T$ & 0.161 & 0.230  & 0.281 & 0.952  & 0.155 & 0.160 & 0.223 & 0.896 & 0.147 & 0.112 & 0.185 & 0.794 \\ 
			$\beta_{C,0}$ & 0.539 & 0.272 & 0.603 & 0.509 & 0.531 & 0.191 & 0.564 & 0.188 & 0.522 & 0.138 & 0.540 & 0.020 \\ 
			$\beta_{C,1}$ & -0.111 & 0.205 & 0.233 & 0.888 & -0.115 & 0.144 & 0.184 & 0.854 & -0.116 & 0.103 & 0.155 & 0.772 \\ 
			$\alpha_C$ & -0.059 & 0.137 & 0.149 & 0.956 & -0.059 & 0.097 & 0.114 & 0.938 & -0.052 & 0.069  & 0.086 & 0.912 \\ 
			$\lambda_C$ & -0.128 & 0.151 & 0.198 & 0.816 & -0.128 & 0.107 & 0.166 & 0.725 & -0.131 & 0.077  & 0.152 & 0.554\\ 
			$\sigma_T$ & 0.019 & 0.070 & 0.073 & 0.941 & 0.026 & 0.051 & 0.057 & 0.908 & 0.033 & 0.036  & 0.049 & 0.844 \\ 
			$\sigma_C$ & 0.069 & 0.093 & 0.116 & 0.872 & 0.082 & 0.063 & 0.103 & 0.748 & 0.087 & 0.046  & 0.098 & 0.515  \\ 
			\hline
			\multicolumn{13}{|c|}{oracle estimator} \\
			\hline
			$\beta_{T,0}$ & 0.000 & 0.154  & 0.154 & 0.952 & 0.001 & 0.109  & 0.109 & 0.944 & 0.001 & 0.076  & 0.076 & 0.950 \\ 
			$\beta_{T,1}$ & 0.001 & 0.100  & 0.100 & 0.949 & -0.000 & 0.069  & 0.069 & 0.954 & -0.001 & 0.049  & 0.049 & 0.946 \\ 
			$\alpha_T$ & -0.001 & 0.066  & 0.066 & 0.941 & -0.001 & 0.045 & 0.045 & 0.952 & 0.001 & 0.032  & 0.032 & 0.956 \\ 
			$\lambda_T$  & -0.000 & 0.088  & 0.088 & 0.946 & 0.001 & 0.061  & 0.061 & 0.950 & 0.000 & 0.043  & 0.043 & 0.950 \\ 
			$\beta_{C,0}$ & -0.003 & 0.210  & 0.210 & 0.941 & 0.000 & 0.147  & 0.147 & 0.946 & -0.000 & 0.102  & 0.102 & 0.954 \\ 
			$\beta_{C,1}$ & -0.001 & 0.122  & 0.122 & 0.946 & -0.000 & 0.085  & 0.085 & 0.946 & 0.000 & 0.060  & 0.060 & 0.950 \\ 
			$\alpha_C$ & 0.002 & 0.078  & 0.078 & 0.946 & -0.000 & 0.056  & 0.056 & 0.945 & 0.000 & 0.040  & 0.040 & 0.942 \\ 
			$\lambda_C$ & -0.002 & 0.104  & 0.104 & 0.949 & 0.000 & 0.074  & 0.074 & 0.950 & -0.001 & 0.054  & 0.054 & 0.936 \\ 
			$\sigma_T$ & -0.010 & 0.063  & 0.064 & 0.950 & -0.007 & 0.044  & 0.044 & 0.946 & -0.003 & 0.032  & 0.032 & 0.942 \\ 
			$\sigma_C$ & -0.019 & 0.079  & 0.081 & 0.939 & -0.010 & 0.055  & 0.056 & 0.944 & -0.004 & 0.039  & 0.040 & 0.950 \\ 
			$\rho$ & 0.001 & 0.099  & 0.099 & 0.939 & -0.002 & 0.067  & 0.067 & 0.943 & -0.001 & 0.046  & 0.046 & 0.952 \\ 
			\hline
			\multicolumn{13}{|c|}{two-step estimator} \\
			\hline
			$\beta_{T,0}$ & 0.017 & 0.385  & 0.385 & 0.949 & 0.014 & 0.279  & 0.279 & 0.949 & 0.005 & 0.190  & 0.190 & 0.952 \\ 
			$\beta_{T,1}$ & 0.010 & 0.299  & 0.299 & 0.946 & 0.003 & 0.213  & 0.212 & 0.950 & 0.002 & 0.148  & 0.148 & 0.950 \\ 
			$\alpha_T$ & -0.016 & 0.213  & 0.214 & 0.946 & -0.012 & 0.147  & 0.148 & 0.945 & -0.006 & 0.099  & 0.099 & 0.962 \\ 
			$\lambda_T$ & 0.014 & 0.219  & 0.219 & 0.949 & 0.012 & 0.154  & 0.155 & 0.946 & 0.006 & 0.103  & 0.103 & 0.956 \\ 
			$\beta_{C,0}$ & 0.008 & 0.300  & 0.300 & 0.949 & 0.008 & 0.212  & 0.212 & 0.950 & 0.003 & 0.146  & 0.146 & 0.953 \\ 
			$\beta_{C,1}$ & 0.004 & 0.209  & 0.209 & 0.940 & 0.002 & 0.150  & 0.150 & 0.932 & 0.003 & 0.103  & 0.103 & 0.950 \\ 
			$\alpha_C$ & -0.008 & 0.145  & 0.145 & 0.942 & -0.007 & 0.102  & 0.103 & 0.944 & -0.003 & 0.069  & 0.069 & 0.957 \\ 
			$\lambda_C$ & 0.008 & 0.162  & 0.163 & 0.953 & 0.007 & 0.112  & 0.112 & 0.947 & 0.003 & 0.078  & 0.078 & 0.950 \\ 
			$\sigma_T$ & -0.010 & 0.063  & 0.064 & 0.942 & -0.007 & 0.044  & 0.044 & 0.942 & -0.003 & 0.032  & 0.032 & 0.937 \\ 
			$\sigma_C$ & -0.019 & 0.079  & 0.081 & 0.938 & -0.010 & 0.055  & 0.056 & 0.939 & -0.004 & 0.039  & 0.040 & 0.948 \\ 
			$\rho$ & 0.001 & 0.099 & 0.099 & 0.928 & -0.002 & 0.067  & 0.067 & 0.938 & -0.001 & 0.046  & 0.046 & 0.950 \\ 
			\hline
		\end{tabular}
	}
\end{table}
\begin{table}[ht]
	\centering
	\caption{Estimation results for design 2 with 46\% censoring and 2500 simulations. Given are the bias, the empirical standard deviation (ESD), the root mean squared error (RMSE) and the confidence rate (CR).}
	\resizebox{\textwidth}{!}{
		\begin{tabular}{|c|rrrr|rrrr|rrrr|}
			\hline
			\multicolumn{5}{|c|}{$n=250$} & \multicolumn{4}{c|}{$n=500$} & \multicolumn{4}{c|}{$n=1000$}\\
			\hline
			\multicolumn{13}{|c|}{naive estimator} \\
			\hline
			& Bias & ESD & RMSE & CR & Bias & ESD & RMSE & CR & Bias & ESD & RMSE & CR \\ 
			\hline
			$\beta_{T,0}$ & 3.072 & 0.726  & 3.156 & 0.010 & 3.061 & 0.451  & 3.094 & 0.000 & 3.075 & 0.304  & 3.090 & 0.000 \\ 
			$\beta_{T,1}$ & 0.368 & 0.257  & 0.449 & 0.602 & 0.359 & 0.182 & 0.403 & 0.397 & 0.362 & 0.125  & 0.383 & 0.129 \\
			$\alpha_T$ & -4.387 & 0.637  & 4.433 & 0.003 & -4.399 & 0.357 & 4.414 & 0.000 & -4.413 & 0.239  & 4.419 & 0.000 \\ 
			$\beta_{C,0}$ & 1.953 & 0.388  & 1.991 & 0.006 & 1.974 & 0.190 & 1.983 & 0.000 & 1.977 & 0.116  & 1.980 & 0.000  \\ 
			$\beta_{C,1}$ & 0.283 & 0.262  & 0.386 & 0.742 & 0.304 & 0.195  & 0.361 & 0.564 & 0.303 & 0.135  & 0.331 & 0.335 \\ 
			$\alpha_C$ & -2.810 & 0.504  & 2.855 & 0.002 & -2.872 & 0.327  & 2.890 & 0.000 & -2.869 & 0.217  & 2.877 & 0.000 \\ 
			$\sigma_T$ & 0.489 & 0.132  & 0.507 & 0.000 & 0.474 & 0.070  & 0.479 & 0.000 & 0.473 & 0.046  & 0.475 & 0.000 \\ 
			$\sigma_C$ & 0.188 & 0.129  & 0.228 & 0.494 & 0.178 & 0.069  & 0.191 & 0.153 & 0.180 & 0.046  & 0.185 & 0.009 \\
			$\rho$ & 0.013 & 0.253 & 0.253 & 0.947 & 0.055 & 0.165 & 0.174 & 0.976 & 0.070 & 0.107 & 0.127 & 0.987 \\
			\hline
			\multicolumn{13}{|c|}{independent estimator} \\
			\hline
			$\beta_{T,0}$ & 0.455 & 0.720 & 0.852 & 0.820 & 0.482 & 0.478  & 0.678 & 0.770 & 0.496 & 0.346  & 0.605 & 0.640 \\ 
			$\beta_{T,1}$ & 0.233 & 0.209  & 0.313 & 0.748 & 0.239 & 0.143  & 0.279 & 0.578 & 0.237 & 0.100  & 0.257 & 0.358 \\ 
			$\alpha_T$ & 0.037 & 0.928  & 0.928 & 0.938 & -0.003 & 0.611  & 0.611 & 0.944 & -0.017 & 0.441  & 0.441 & 0.937 \\  
			$\lambda_T$& 0.256 & 0.354  & 0.437 & 0.959 & 0.236 & 0.234  & 0.332 & 0.909 & 0.230 & 0.169  & 0.285 & 0.783  \\ 
			$\beta_{C,0}$ & 0.557 & 0.515  & 0.758 & 0.738 & 0.547 & 0.358  & 0.654 & 0.631 & 0.536 & 0.253 & 0.593 & 0.414 \\ 
			$\beta_{C,1}$  & -0.300 & 0.169  & 0.345 & 0.568 & -0.294 & 0.118 & 0.316 & 0.275 & -0.293 & 0.082 & 0.305 & 0.044 \\ 
			$\alpha_C$ & 0.012 & 0.679  & 0.679 & 0.949 & 0.015 & 0.474 & 0.474 & 0.944 & 0.029 & 0.331 & 0.333 & 0.956  \\ 
			$\lambda_C$ & -0.228 & 0.253 & 0.340 & 0.787 & -0.223 & 0.177  & 0.285 & 0.712 & -0.216 & 0.122  & 0.248 & 0.562 \\ 
			$\sigma_T$ & 0.052 & 0.071  & 0.088 & 0.874 & 0.057 & 0.050  & 0.076 & 0.778 & 0.060 & 0.036  & 0.070 & 0.580 \\ 
			$\sigma_C$ & 0.197 & 0.105 & 0.224 & 0.478 & 0.205 & 0.072  & 0.217 & 0.150 & 0.211 & 0.054  & 0.218 & 0.010  \\ 
			\hline
			\multicolumn{13}{|c|}{oracle estimator} \\
			\hline
			$\beta_{T,0}$ & -0.003 & 0.300  & 0.300 & 0.944 & -0.000 & 0.207  & 0.207 & 0.950 & 0.003 & 0.146  & 0.146 & 0.950  \\ 
			$\beta_{T,1}$ & -0.003 & 0.119  & 0.119 & 0.941 & -0.002 & 0.084  & 0.084 & 0.944 & -0.001 & 0.057  & 0.057 & 0.956 \\ 
			$\alpha_T$ & -0.000 & 0.356  & 0.356 & 0.950 & -0.003 & 0.251  & 0.251 & 0.946 & -0.004 & 0.174  & 0.174 & 0.947 \\  
			$\lambda_T$ & -0.001 & 0.166  & 0.166 & 0.948 & -0.000 & 0.118  & 0.118 & 0.942 & -0.001 & 0.081  & 0.081 & 0.957 \\ 
			$\beta_{C,0}$ & 0.008 & 0.388  & 0.388 & 0.945& -0.000 & 0.264  & 0.264 & 0.950 & 0.005 & 0.187  & 0.187 & 0.954 \\ 
			$\beta_{C,1}$ & 0.005 & 0.150  & 0.150 & 0.949 & 0.002 & 0.104  & 0.104 & 0.945 & 0.001 & 0.075  & 0.075 & 0.946 \\ 
			$\alpha_C$ & -0.017 & 0.450  & 0.450 & 0.944 & -0.007 & 0.306  & 0.306 & 0.947 & -0.007 & 0.220  & 0.220 & 0.949\\ 
			$\lambda_C$ & -0.001 & 0.192  & 0.192 & 0.945 & 0.001 & 0.133  & 0.133 & 0.944 & -0.001 & 0.093  & 0.093 & 0.946\\ 
			$\sigma_T$ & -0.006 & 0.058  & 0.059 & 0.948 & -0.004 & 0.041  & 0.041 & 0.948 & -0.002 & 0.030  & 0.030 & 0.942\\ 
			$\sigma_C$ & -0.018 & 0.093  & 0.095 & 0.931 & -0.010 & 0.066  & 0.067 & 0.939 & -0.004 & 0.047  & 0.047 & 0.944\\ 
			$\rho$ & -0.011 & 0.146  & 0.146 & 0.972 & -0.005 & 0.103  & 0.103 & 0.960 & -0.004 & 0.071  & 0.071 & 0.961 \\ 
			\hline
			\multicolumn{13}{|c|}{two-step estimator} \\
			\hline
			$\beta_{T,0}$ & -0.064 & 0.619  & 0.622 & 0.957 & -0.019 & 0.428  & 0.428 & 0.948 & -0.009 & 0.301  & 0.301 & 0.948 \\ 
			$\beta_{T,1}$ & -0.009 & 0.205  & 0.205 & 0.944 & -0.009 & 0.142  & 0.143 & 0.952 & -0.004 & 0.102  & 0.102 & 0.944 \\ 
			$\alpha_T$ & 0.071 & 0.804  & 0.807 & 0.940 & 0.019 & 0.556  & 0.556 & 0.942 & 0.008 & 0.386  & 0.386 & 0.952 \\  
			$\lambda_T$ & 0.023 & 0.325  & 0.326 & 0.943 & 0.006 & 0.225  & 0.225 & 0.940 & 0.003 & 0.155  & 0.155 & 0.954 \\ 
			$\beta_{C,0}$ & -0.013 & 0.506  & 0.506 & 0.950 & -0.005 & 0.341  & 0.341 & 0.949 & 0.001 & 0.245  & 0.245 & 0.950 \\ 
			$\beta_{C,1}$ & 0.001 & 0.181  & 0.181 & 0.947 & -0.001 & 0.124  & 0.124 & 0.948 & -0.001 & 0.088  & 0.088 & 0.948 \\ 
			$\alpha_C$ & 0.010 & 0.635  & 0.635 & 0.946 & -0.001 & 0.429  & 0.429 & 0.948 & -0.003 & 0.307 & 0.307 & 0.949 \\  
			$\lambda_C$ & 0.004 & 0.252  & 0.252 & 0.945 & 0.002 & 0.172  & 0.172 & 0.956 & -0.000 & 0.123  & 0.122 & 0.947 \\  
			$\sigma_T$ & -0.002 & 0.059  & 0.059 & 0.948 & -0.002 & 0.041  & 0.041 & 0.950 & -0.001 & 0.030  & 0.030 & 0.940 \\ 
			$\sigma_C$ & -0.015 & 0.094  & 0.095 & 0.926 & -0.009 & 0.067  & 0.068 & 0.936 & -0.003 & 0.047  & 0.047 & 0.940 \\ 
			$\rho$ & -0.013 & 0.150  & 0.150 & 0.970 & -0.004 & 0.104  & 0.104 & 0.962 & -0.004 & 0.073  & 0.073 & 0.954 \\ 
			\hline
		\end{tabular}
	}
\end{table}
\begin{table}[ht]
	\centering
	\caption{Estimation results for design 3 with 46\% censoring and 2500 simulations. Given are the bias, the empirical standard deviation (ESD), the root mean squared error (RMSE) and the confidence rate (CR).}
	\resizebox{\textwidth}{!}{
		\begin{tabular}{|c|rrrr|rrrr|rrrr|}
			\hline
			\multicolumn{5}{|c|}{$n=250$} & \multicolumn{4}{c|}{$n=500$} & \multicolumn{4}{c|}{$n=1000$}\\
			\hline
			\multicolumn{13}{|c|}{naive estimator} \\
			\hline
			& Bias & ESD & RMSE & CR & Bias & ESD & RMSE & CR & Bias & ESD & RMSE & CR \\ 
			\hline
			$\beta_{T,0}$ & -0.226 & 0.244 & 0.333 & 0.815 & -0.231 & 0.167 & 0.285 & 0.714 & -0.233 & 0.119 & 0.262 & 0.509 \\ 
			$\beta_{T,1}$ & -0.902 & 0.163  & 0.916 & 0.001 & -0.902 & 0.116 & 0.910 & 0.000 & -0.903 & 0.082 & 0.906 & 0.000  \\ 
			$\alpha_T$ & 1.499 & 0.102 & 1.502 & 0.000 & 1.499 & 0.073 & 1.500 & 0.000 & 1.497 & 0.051 & 1.498 & 0.000 \\ 
			$\beta_{C,0}$ & -0.144 & 0.279  & 0.314 & 0.914 & -0.148 & 0.195 & 0.245 & 0.886 & -0.153 & 0.138 & 0.206 & 0.801 \\ 
			$\beta_{C,1}$ & -0.541 & 0.140  & 0.559 & 0.024 & -0.542 & 0.097 & 0.550 & 0.000 & -0.539 & 0.069 & 0.543 & 0.000 \\ 
			$\alpha_C$ & 0.905 & 0.103  & 0.911 & 0.000 & 0.906 & 0.073 & 0.909 & 0.000 & 0.909 & 0.052 & 0.910 & 0.000  \\ 
			$\sigma_T$ & 1.165 & 0.107  & 1.170 & 0.000 & 1.172 & 0.075 & 1.175 & 0.000 & 1.176 & 0.054 & 1.178 & 0.000 \\ 
			$\sigma_C$ & 0.425 & 0.096  & 0.435 & 0.002 & 0.432 & 0.068 & 0.437 & 0.000 & 0.439 & 0.049 & 0.442 & 0.000 \\ 
			$\rho$ & 0.091 & 0.077  & 0.119 & 0.860 & 0.093 & 0.054 & 0.108 & 0.716 & 0.097 & 0.036 & 0.104 & 0.443 \\  
			\hline
			\multicolumn{13}{|c|}{independent estimator} \\
			\hline
			$\beta_{T,0}$ & 0.444 & 0.309  & 0.541 & 0.712 & 0.429 & 0.215 & 0.480 & 0.481 & 0.431 & 0.151  & 0.456 & 0.180 \\ 
			$\beta_{T,1}$ & 0.135 & 0.331  & 0.357 & 0.954 & 0.118 & 0.230 & 0.258 & 0.939 & 0.116 & 0.157  & 0.195 & 0.906 \\ 
			$\alpha_T$ & 0.016 & 0.258 & 0.259 & 0.928 & 0.035 & 0.177 & 0.180 & 0.927 & 0.043 & 0.124  & 0.131 & 0.914 \\ 
			$\lambda_T$ & 0.160 & 0.263  & 0.308 & 0.968 & 0.141 & 0.180 & 0.229 & 0.939 & 0.134 & 0.127  & 0.185 & 0.869 \\ 
			$\beta_{C,0}$ & 0.571 & 0.213  & 0.610 & 0.216 & 0.561 & 0.147 & 0.580 & 0.024 & 0.560 & 0.106  & 0.570 & 0.000 \\ 
			$\beta_{C,1}$ & -0.121 & 0.217 & 0.249 & 0.883 & -0.133 & 0.150 & 0.201 & 0.819 & -0.129 & 0.107  & 0.167 & 0.756 \\ 
			$\alpha_C$ & -0.076 & 0.161  & 0.178 & 0.961 & -0.065 & 0.111 & 0.129 & 0.947 & -0.063 & 0.081  & 0.103 & 0.901 \\ 
			$\lambda_C$ & -0.139 & 0.175 & 0.223 & 0.821 & -0.148 & 0.121 & 0.191 & 0.740 & -0.146 & 0.088  & 0.171 & 0.578 \\ 
			$\sigma_T$  & 0.017 & 0.066  & 0.068 & 0.944 & 0.023 & 0.048 & 0.053 & 0.922 & 0.028 & 0.034 & 0.045 & 0.860 \\ 
			$\sigma_C$ & 0.080 & 0.099 & 0.127 & 0.854 & 0.094 & 0.068 & 0.116 & 0.711 & 0.100 & 0.049 & 0.111 & 0.454 \\ 
			\hline
			\multicolumn{13}{|c|}{oracle estimator} \\
			\hline
			$\beta_{T,0}$ & -0.003 & 0.130  & 0.130 & 0.946 & -0.001 & 0.092  & 0.092 & 0.953 & 0.001 & 0.064  & 0.064 & 0.950 \\ 
			$\beta_{T,1}$ & 0.002 & 0.099  & 0.099 & 0.948 & -0.002 & 0.069  & 0.069 & 0.948 & -0.001 & 0.048  & 0.048 & 0.952\\ 
			$\alpha_T$ & -0.003 & 0.072  & 0.072 & 0.946 & 0.001 & 0.050  & 0.050 & 0.955 & 0.000 & 0.036  & 0.036 & 0.953\\ 
			$\lambda_T$  & 0.001 & 0.090  & 0.090 & 0.950 & -0.001 & 0.063  & 0.063 & 0.952 & 0.000 & 0.044  & 0.044 & 0.949\\ 
			$\beta_{C,0}$ & -0.002 & 0.188  & 0.188 & 0.946 & 0.000 & 0.131  & 0.131 & 0.956 & 0.000 & 0.092  & 0.092 & 0.948 \\ 
			$\beta_{C,1}$ & -0.002 & 0.130  & 0.130 & 0.948 & 0.000 & 0.089  & 0.089 & 0.953 & -0.000 & 0.065  & 0.065 & 0.946 \\ 
			$\alpha_C$ & 0.001 & 0.092  & 0.092 & 0.953 & -0.000 & 0.066  & 0.066 & 0.951 & 0.000 & 0.046  & 0.046 & 0.954 \\ 
			$\lambda_C$ & -0.001 & 0.121  & 0.121 & 0.946 & -0.000 & 0.081  & 0.081 & 0.954 & -0.001 & 0.060  & 0.060 & 0.937 \\ 
			$\sigma_T$ & -0.010 & 0.061  & 0.061 & 0.942 & -0.006 & 0.043  & 0.043 & 0.943 & -0.003 & 0.031  & 0.031 & 0.938 \\ 
			$\sigma_C$ & -0.020 & 0.082  & 0.084 & 0.940 & -0.011 & 0.058  & 0.059 & 0.942 & -0.005 & 0.041  & 0.041 & 0.953 \\ 
			$\rho$ & 0.002 & 0.098  & 0.098 & 0.937 & -0.001 & 0.067  & 0.067 & 0.942 & -0.001 & 0.046  & 0.046 & 0.950 \\ 
			\hline
			\multicolumn{13}{|c|}{two-step estimator} \\
			\hline
			$\beta_{T,0}$ & -0.002 & 0.291  & 0.290 & 0.951 & -0.001 & 0.204  & 0.204 & 0.952 & -0.002 & 0.141  & 0.141 & 0.954 \\ 
			$\beta_{T,1}$ & 0.018 & 0.310  & 0.310 & 0.946 & 0.000 & 0.211  & 0.211 & 0.949 & 0.002 & 0.150  & 0.150 & 0.953 \\ 
			$\alpha_T$ & -0.029 & 0.243  & 0.245 & 0.954 & -0.011 & 0.167  & 0.167 & 0.950 & -0.005 & 0.115  & 0.116 & 0.956 \\ 
			$\lambda_T$ & 0.027 & 0.247  & 0.248 & 0.952 & 0.011 & 0.171  & 0.172 & 0.954 & 0.006 & 0.117  & 0.117 & 0.957 \\ 
			$\beta_{C,0}$ & -0.001 & 0.244  & 0.244 & 0.949 & 0.000 & 0.170  & 0.170 & 0.952 & -0.001 & 0.120  & 0.120 & 0.955 \\ 
			$\beta_{C,1}$ & 0.009 & 0.222  & 0.223 & 0.944 & 0.002 & 0.151  & 0.151 & 0.952 & 0.002 & 0.108  & 0.108 & 0.944 \\ 
			$\alpha_C$ & -0.015 & 0.171  & 0.172 & 0.950 & -0.007 & 0.116  & 0.116 & 0.952 & -0.003 & 0.080  & 0.080 & 0.958 \\ 
			$\lambda_C$ & 0.016 & 0.189  & 0.190 & 0.948 & 0.007 & 0.125  & 0.125 & 0.952 & 0.003 & 0.089  & 0.089 & 0.951 \\ 
			$\sigma_T$ & -0.010 & 0.061  & 0.061 & 0.940 & -0.006 & 0.043  & 0.043 & 0.938 & -0.003 & 0.031  & 0.031 & 0.936 \\ 
			$\sigma_C$ & -0.020 & 0.082  & 0.084 & 0.938 & -0.011 & 0.058  & 0.059 & 0.941 & -0.005 & 0.041  & 0.041 & 0.948 \\ 
			$\rho$ & 0.002 & 0.098  & 0.098 & 0.924 & -0.001 & 0.067  & 0.067 & 0.938 & -0.001 & 0.046  & 0.046 & 0.949 \\ 
			\hline
		\end{tabular}
	}
\end{table}
\end{document}